\documentclass[11pt, reqno]{amsart}

\usepackage{graphicx}
\usepackage{amsmath,mathtools,amsopn,amsfonts,amsthm}
\usepackage{bbm}
\usepackage{subfigure}
\usepackage{eepic}
\usepackage{a4wide}
\usepackage[utf8]{inputenc}
\usepackage{epsfig}
\usepackage{latexsym}
\usepackage[shortlabels]{enumitem}
\usepackage[UKenglish]{babel}
\usepackage{verbatim}
\usepackage{color}
\usepackage[usenames,dvipsnames,table,xcdraw]{xcolor}
\usepackage{pgf,tikz,pgfplots}
\pgfplotsset{compat=1.15}
\usepackage{mathrsfs}
\usetikzlibrary{arrows}
\usepackage{booktabs}
\usepackage[initials]{amsrefs}
\usepackage[colorlinks=true]{hyperref}
\usepackage{longtable}
\usepackage{float}
\usepackage{wrapfig}
\usepackage{caption}
\usepackage{listings}

\newtheorem{thm}{Theorem}[section]
\newtheorem{lma}[thm]{Lemma}

\newtheorem{cor}[thm]{Corollary}
\newtheorem{prop}[thm]{Proposition}
\newtheorem{claim}[thm]{Claim}
\theoremstyle{definition}

\newtheorem{rem}[thm]{Remark}
\newtheorem{ques}[thm]{Question}

\numberwithin{equation}{section}

\newcommand{\R}{\mathbb{R}}

\newcommand{\N}{\mathbb{N}}

\newcommand{\p}{\mathbf{p}}
\renewcommand{\q}{\mathbf{q}}
\providecommand{\norm}[1]{\lVert#1\rVert}
\newcommand{\E}{\mathbb{E}}

\newcommand{\I}{\mathcal{I}}

\renewcommand{\i}{\ii}

\newcommand{\ii}{\mathbf{i}}
\newcommand{\jj}{\mathbf{j}}

\newcommand{\ih}{\hat\imath}
\newcommand{\jh}{\hat\jmath}
\newcommand{\iiv}{\overline{\imath}}
\newcommand{\jjv}{\overline{\jmath}}
\newcommand{\rmu}{\overleftarrow{\mu}}
\newcommand{\rnu}{\overleftarrow{\nu}}

\definecolor{qqqqff}{rgb}{0,0,1}
\definecolor{wwwwww}{rgb}{0.4,0.4,0.4}
\definecolor{ccqqqq}{rgb}{0.8,0,0}
\definecolor{ffffff}{rgb}{1,1,1}

\title{On the convergence rate of the chaos game}

\author{Bal\'azs B\'ar\'any}
\address{Bal\'azs B\'ar\'any; Budapest University of Technology and Economics, MTA-BME Stochastics Research Group, P.O. Box 91, 1521 Budapest, Hungary} \email{balubs@math.bme.hu}

\author{Natalia Jurga} \address{Natalia Jurga; University of St Andrews,  School of Mathematics and Statistics, \newline St Andrews, KY16 9SS, Scotland}
\email{naj1@st-andrews.ac.uk}

\author{Istv\'an Kolossv\'ary}
\address{Istv\'an Kolossv\'ary; University of St Andrews,  School of Mathematics and Statistics, \newline St Andrews, KY16 9SS, Scotland} \email{itk1@st-andrews.ac.uk}

\thanks{Bal\'azs B\'ar\'any acknowledges support from grants OTKA K123782 and OTKA FK134251. 759/1). Natalia Jurga was supported by an EPSRC Standard Grant (EP/R015104/1). Istv\'an Kolossv\'ary was supported by a \emph{Leverhulme Trust Research Project Grant} (RPG-2019-034).}

\begin{document}

\begin{abstract}
This paper studies how long it takes the orbit of the chaos game to reach a certain density inside the attractor of a strictly contracting iterated function system of which we only assume that its lower dimension is positive. We show that the rate of growth of this cover time is determined by the Minkowski dimension of the push-forward of the shift invariant measure with exponential decay of correlations driving the chaos game. Moreover, we bound the expected value of the cover time from above and below with multiplicative logarithmic correction terms. As an application, for Bedford-McMullen carpets we completely characterise the family of probability vectors which minimise the Minkowski dimension of Bernoulli measures. Interestingly, these vectors have not appeared in any other aspect of Bedford-McMullen carpets before.
\end{abstract}

 \maketitle

\thispagestyle{empty}

\section{Introduction}

The \emph{chaos game} is a simple random iterative procedure introduced by Barnsley~\cite{BarnsleyBook} to generate the  attractor of an \emph{iterated function system} (IFS) which is a tuple $\mathcal{F}=\{f_1,f_2,\ldots,f_N\}$  of contracting transformations on $\R^d$. Given a fixed non-degenerate probability vector $\p=(p_1,\ldots,p_N)$, the game starts from an initial point $x_0$, chooses a random function $f_{i_1}$ from $\mathcal{F}$ according to $\p$ and returns the point $x_1=f_{i_1}(x_0)$. It continues iteratively generating a sequence of points $x_2,x_3,\ldots$, where $x_n=f_{i_n}(x_{n-1})$ and the indices $i_n$ are independently and identically distributed (i.i.d.) according to~$\p$. The \emph{attractor} $\Lambda\subseteq\R^d$, which is the unique non-empty, compact set satisfying $\Lambda = \bigcup_{i=1}^Nf_i(\Lambda)$ is then almost surely (regardless of the starting point) obtained as the $\omega$-limit set of the orbit $(x_n)_{n=0}^{\infty}$, i.e. $\Lambda=\bigcap_{m=1}^\infty\overline{\{ x_n:\, n\geq m \}}$. The almost sure property is with respect to the Bernoulli measure $\mu_{\p}=\p^\N$ defined on infinite sequences $\ii$ from the symbolic space $\Sigma=\{1,2,\ldots,N\}^{\N}$. In particular, if $x_0\in \Lambda$ (for example, $x_0$ is a fixed point of one of the $f_i$) then $(x_n)_{n=0}^{\infty}$ is almost surely a dense subset of $\Lambda$. For the remainder we will always assume that $x_0\in \Lambda$.

The chaos game can be naturally generalised to allow the indices $i_n$ to be chosen in a non- i.i.d. way. The left shift $\sigma:\, \Sigma\to\Sigma$ is defined $\sigma(i_1i_2i_3\ldots)=i_2i_3\ldots$ and a measure $\mu$ on the symbolic space $\Sigma$ is $\sigma$-invariant if $\mu(\sigma^{-1}(\cdot))=\mu(\cdot)$. Then we also consider the chaos game with respect to a $\sigma$-invariant measure $\mu$, where for every $n \in \N$ and $i_1, \ldots, i_n \in \{1, \ldots, N\}^n$, the first $n$ indices in the chaos game are chosen to be $i_1, \ldots, i_n$ with probability $\mu([i_1 \ldots i_n])$, where $[i_1 \ldots i_n]$ denotes the cylinder set $\{\jj\in\Sigma:\, j_l=i_l \text{ for }l=1,\ldots,n\}$.

The natural question we address in this paper is: given $\mu$, how long does it take the chaos game to reach a certain `resolution' of the fractal? Moreover, for which $\mu$ will it take the least number of steps to reach that resolution? We formulate these questions rigorously in a moment. We are only aware of \cite{GutierrezEtal_Fractals96} and \cite{JurgaMorris_ChaosGame_2020arxiv} that study this direction. The former is an informal investigation, while the latter work of Morris and the second named author answer these questions for self-similar sets satisfying the open set condition in the i.i.d case. Our results greatly extend this scope and give additional context explaining and illustrating the results. A number of papers show convergence of the chaos game for more general IFSs, see~\cites{BARNSLEYetal_JMAA2016,BarnsleyLesniak_2014,barnsley_vince_2011,Lesniak_chaos2015} for example, while others show ergodic theorems for time averages along the orbit~\cites{elton_1987,forte_mendivil_1998,Werner_2004}. We do not pursue these directions further here.

The question addressed in this article is a type of ``covering problem'', which are problems with a rich history in the probability literature, originating with the classical coupon collector problem. Covering problems are concerned with the time taken for a stochastic process to ``exhaust'' its state space, in a precise sense which depends on the model being studied. Classical examples include the first time that a random walk on a finite path-connected graph visits all of the vertices \cite{aldous-graph} and the first time that an irreducible finite-state Markov chain visits all of its states \cite{peres, matthews}. The geometric covering problems studied in \cite{aldous, penrose} are closely related to our problem, except that while our sequence of points $x_1,x_2,x_3 \ldots$ is determined in a dynamical way through the chaos game, the points belonging to the sequences in \cite{aldous, penrose} are chosen at random in an i.i.d. way.  Another closely related field is the study of recurrence in chaotic dynamical systems through hitting time statistics \cite{mike-book, haydn}. This is concerned with the study of the hitting time of an orbit under a chaotic dynamical system, to a ball which is shrinking down to a point in the state space.

\subsection*{Setup} Throughout, we assume that the IFS $\mathcal{F}$ is strictly contracting in the sense that there exists $0<a<1$ such that
\begin{equation}\label{eq:condition}
\|f_i(x)-f_i(y))\|\leq a\|x-y\|\text{ for every }i=1,\ldots,N.
\end{equation}
A priori we do not assume any separation condition on the first level cylinder sets $f_i(\Lambda)$ nor any further smoothness conditions on the maps $f_i$.

The \emph{orbit} of the starting point $x_0\in\Lambda$ according to $\ii\in\Sigma$ for the first $n$ steps is
\begin{equation*}
\mathcal{O}_n(\ii,x_0) =\big\{x_0, f_{i_1}(x_0), f_{i_2}\circ f_{i_1}(x_0),\ldots,f_{i_n}\circ\ldots\circ f_{i_2}\circ f_{i_1}(x_0)\big\}.
\end{equation*}
For compositions of maps we use the standard notation $f_{\ii|n}:=f_{i_1}\circ f_{i_2}\circ\ldots\circ f_{i_n}$, where $\ii|n = i_1i_2\ldots i_n$. Compositions are taken in reverse order in $\mathcal{O}_n(\ii,x_0)$, for this we introduce the notation $\overleftarrow{\ii|n}:=i_ni_{n-1}\ldots i_1$ and so $f_{\overleftarrow{\ii|n}} = f_{i_n}\circ f_{i_{n-1}}\circ\ldots\circ f_{i_1}$. We measure the `resolution' of $\mathcal{O}_n(\ii,x_0)$ as a subset of $\Lambda$ according to the Hausdorff distance $d_{\mathrm{H}}$ between them. In particular, for any $r>0$
\begin{equation*}
d_{\mathrm H} \big(\mathcal{O}_n(\ii,x_0),\Lambda\big) < r \;\Longleftrightarrow\; (\forall x\in\Lambda)\, (\exists y\in\mathcal{O}_n(\ii,x_0)) \text{ such that } x\in B(y,r).
\end{equation*}
In this case, we also say that $\mathcal{O}_n(\ii,x_0)$ is $r$-\emph{dense} in $\Lambda$. To measure the first instance when $\mathcal{O}_n(\ii,x_0)$ becomes $r$-dense in $\Lambda$, we introduce the waiting time
\begin{equation*}
T_r(\ii,x_0) := \inf \{n\geq 0:\, d_{\mathrm H} \big(\mathcal{O}_n(\ii,x_0),\Lambda\big) < r\}.
\end{equation*}

Let $\mu$ be a left-shift invariant measure. We say that $\mu$ has {\it one-sided exponential decay} if there exists $\varepsilon>0$ and $\kappa>0$ such that for any finite words $\iiv,\jjv\in\Sigma^*$
\begin{equation}\label{eq:decay}
\mu([\iiv]\cap\sigma^{-n}[\jjv])\leq (1+\kappa2^{-\varepsilon(n-|\iiv|)})\mu([\iiv])\mu([\jjv]).
\end{equation}
Such a measure $\mu$ is necessarily strongly mixing. To see this, note that \eqref{eq:decay} immediately implies that for all finite words $\iiv,\jjv\in\Sigma^*$, $\mu$ satisfies the one-sided strong mixing property $\limsup_{n \to \infty} \mu([\iiv]\cap\sigma^{-n}[\jjv]) \leq \mu([\iiv])\mu([\jjv])$. The other side can then be obtained by applying the one-sided strong mixing property to $\mu((\Sigma \setminus [\iiv])\cap\sigma^{-n}[\jjv])).$ It is easy to see that Bernoulli measures satisfy the one-sided exponential decay property \eqref{eq:decay}, but it is also satisfied by many other measures, such as Gibbs measures for H\"older continuous potentials on $\Sigma$ and certain K\"aenm\"aki measures, both of which are important in the analysis of fractals and motivate the study of the chaos game beyond the i.i.d. case. A more detailed discussion of the applicability of our results to these classes of measures can be found in Section \ref{measures}.

We define the {\it reversed measure of a $\sigma$-invariant measure $\mu$}, denoted by $\overleftarrow{\mu}$, as
$$
\rmu([\iiv])=\mu([\overleftarrow{\iiv}])\text{ for every }\iiv\in\Sigma^*.
$$
Since $\mu$ is $\sigma$-invariant it follows that $\rmu$ is a well defined $\sigma$-invariant measure on $\Sigma$. We let $\rnu:=\pi_*\rmu$  where $\pi_*\rmu$ is the pushforward measure of $\rmu$ through the natural projection $\pi:\Sigma\to\Lambda$ defined by
\begin{equation}\label{eq:natproj}
\pi(\ii):=\lim_{n\to \infty} f_{\ii|n} (0).
\end{equation}

\subsection*{Main contribution}
Informally, our main result states that if $\Lambda$ has positive lower dimension $\dim_{\mathrm{L}}\Lambda>0$, see~\eqref{eq:dimL} for a definition,  and $\mu$ is an invariant measure with one-sided exponential decay then $T_r(\ii,x_0)$ asymptotically scales as $r^{-\alpha(1+o(1))}$ as $r\to 0$, where the exponent $\alpha$ is given by the Minkowski dimension of $\rnu$ as defined by Falconer, Fraser and K\"{a}enm\"{a}ki in~\cite{FFF_MinkowskiDimMeas_2020arxiv}, see Section~\ref{def:1} below. We obtain more precise bounds for the expected value of $T_r(\ii,x_0)$ with respect to $\mu$ by showing that
\begin{equation*}
(r(\log(1/r))^{2/\dim_{\mathrm{L}}\Lambda})^{-\alpha+o(1)} \leq \E_\mu T_r(x_0) \leq \left(\log(1/r)\right)^2 r^{-\alpha-o(1)}.
\end{equation*}
Moreover, we give a complete characterisation of the probability vectors $\p$ which minimise the value of $\alpha$ for the pushforward measure $\nu_{\mathbf{p}}$ in the case of Bedford--McMullen carpets. See Section~\ref{sec:results} for the precise formulation of these statements.

\subsection{Minkowski dimension of measures}\label{def:1}
	Let $\nu$ be a fully supported finite Borel measure on a compact metric space $X$. Then the \emph{upper Minkowski dimension} of $\nu$ is defined as
	\begin{align*}
	\overline{\dim}_{\mathrm{M}}(\nu):=\inf \{s \geq 0: & \text { there exists a constant } c>0 \text { such that }\\
	&\;\nu(B(x, r)) \geq c r^{s} \text { for all } x \in X \text { and } 0<r<1\}
	\end{align*}
	and its \emph{lower Minkowski dimension} is
	\begin{align*}
	\underline{\dim}_{\mathrm{M}}(\nu):=\inf \{s \geq 0: & \text { there exist a constant } c>0 \text { and a sequence }\left(r_{n}\right)_{n \in \mathbb{N}}\\
	&\text { of positive real numbers such that } \lim _{n \rightarrow \infty} r_{n}=0 \text { and }\\
	&\;\nu\left(B\left(x, r_{n}\right)\right) \geq c r_{n}^{s} \text { for all } x \in X \text { and } n \in \mathbb{N} \}
	\end{align*}
(here $B(x,r)$ denotes an open ball). 	If the two values coincide, then the common value, called the \emph{Minkowski dimension} of $\nu$, is denoted by $\dim_{\mathrm{M}}\nu$. It is also referred to as the box dimension of $\nu$, see \cite[Chapter 4.2]{fraser_2020Book}.

The upper and lower Minkowski dimension of a set $X$ is denoted by $\overline{\dim}_{\mathrm{M}}X$ and $\underline{\dim}_{\mathrm{M}}X$, respectively. The main result of~\cite{FFF_MinkowskiDimMeas_2020arxiv} is to give an alternative characterization of these quantities by showing that
\begin{equation}\label{eq:100}
\overline{\dim}_{\mathrm{M}}X=\min \left\{\overline{\dim}_{\mathrm{M}}(\nu): \nu \text { is a fully supported finite Borel measure on } X\right\},
\end{equation}
and an analogous claim holds for $\underline{\dim}_{\mathrm{M}}X$. In general, for the attractor of an IFS, the measure that achieves the minimum is not invariant. However, an interpretation of the result in~\cite{JurgaMorris_ChaosGame_2020arxiv} is that for a self-similar set $\Lambda$ satisfying the open set condition, $\dim_{\mathrm{M}}\Lambda$ is achieved by a self-similar measure $\nu_{\p}$ if and only if $\p$ is the so-called `natural measure' associated to the IFS generating $\Lambda$. On the other hand, in the case of a family of self-affine sets called \emph{Bedford--McMullen carpets} there is indeed a positive gap $\dim_{\mathrm{M}}\Lambda < \min_{\mathbf{p}} \dim_{\mathrm{M}} \nu_{\p}$ unless the Hausdorff and Minkowski dimensions of $\Lambda$ coincide, see Remark~\ref{rem:1}. A key contribution of the current paper is to identify the family of vectors $\p$ which minimise $\dim_{\mathrm{M}} \nu_{\p}$. Hence, these are the vectors for which the chaos game reaches a certain `resolution' in the least number of steps.

Here we give a simple, yet useful equivalent characterization of $\dim_{\mathrm{M}}\nu$. Let
\begin{equation*}
\underline{\alpha}(\nu):= \liminf_{r\to 0} \max_{y\in\Lambda} \frac{\log \nu(B(y,r))}{\log r} \;\text{ and }\; \overline{\alpha}(\nu):= \limsup_{r\to 0} \max_{y\in\Lambda} \frac{\log \nu(B(y,r))}{\log r}.
\end{equation*}

We note that $\min_{y \in \Lambda} \nu(B(y,r))$ exists for each $r$ by compactness of $\Lambda$ and lower semi-continuity of $y \mapsto \nu(B(y,r))$, thus the above two definitions make sense.

\begin{lma}\label{lem:equal}
	Let $\nu$ be a compactly supported finite Borel measure. Then
	$$
	\underline{\alpha}(\nu)=\underline{\dim}_{\mathrm{M}}(\nu) \quad\text{ and }\quad	\overline{\alpha}(\nu)=\overline{\dim}_{\mathrm{M}}(\nu).
	$$
\end{lma}

\begin{proof}The proof follows from the definition of the upper and lower Minkowski dimensions. Let $X$ be the support of $\nu$. Let $s>\overline{\dim}_{\mathrm{M}}(\nu)$ be arbitrary. Then for every $x\in X$, $\tfrac{\log\nu(B(x,r))}{\log r}\leq s+\tfrac{\log c}{\log r}$ and so, $\overline{\alpha}(\nu)\leq s$. Since $s$ was arbitrary, we have $\overline{\alpha}(\nu)\leq\overline{\dim}_{\mathrm{M}}(\nu)$.
	
Now, let $s<\overline{\dim}_{\mathrm{M}}(\nu)$. Then by definition for every $c>0$ there exists $x_c\in X$ and $0<r_c<1$ such that $\nu(B(x_c,r_c))<c(r_c)^s$. Hence,
	$$
	\max_{y\in X}\frac{\log\nu(B(y,r_c))}{\log r_c}\geq\frac{\log\nu(B(x_c,r_c))}{\log r_c}>s+\frac{\log c}{\log r_c}.
	$$
Letting $c\to0$, necessarily we have that $r_c\to0$, and since $\tfrac{\log c}{\log r_c}\geq0$, we get	$\overline{\alpha}(\nu)\geq s$. Since $s$ was again arbitrary, we have $\overline{\alpha}(\nu)\geq\overline{\dim}_{\mathrm{M}}(\nu)$.
	
The proof of the other equality is similar. Let $s>\underline{\dim}_{\rm M}(\nu)$. Then there exist $c>0$ and a sequence $r_n$ with $r_n\to0$ such that
	$$
	\liminf_{r\to0}\max_{x\in X}\frac{\log\nu(B(x,r))}{\log r}\leq\liminf_{n\to\infty}\max_{x\in X}\frac{\log\nu(B(x,r_n))}{\log r_n}\leq\liminf_{n\to\infty}\left(s+\frac{\log c}{\log r_n}\right)=s.
	$$
Now, let $s<\underline{\dim}_{\rm M}(\nu)$. Then for every $c>0$ and every sequence $r_n$ with $r_n\to0$ there exists $x\in X$ such that $\nu(B(x,r_n))<c(r_n)^s$. Hence, choosing $r_n$ to be the sequence for which $\underline{\alpha}(\nu)$ is achieved, we get that
	$$
	\underline{\alpha}(\nu)=\lim_{n\to\infty}\max_{y\in X}\frac{\log\nu(B(y,r_n))}{\log r_n}\geq\lim_{n\to\infty} s+\frac{\log c}{\log r_n}=s.
	$$
\end{proof}

The relationship of the Minkowski dimension of $\nu$ to other well-studied notions becomes more apparent from this formulation. Recall that the lower and upper local dimensions of a measure $\nu$ at a point $x$ are
\begin{equation*}
\underline{\dim}_\mathrm{loc}(\nu,x)= \liminf_{r\to 0} \frac{\log \nu(B(x,r))}{\log r} \;\text{ and }\; \overline{\dim}_\mathrm{loc}(\nu,x)= \limsup_{r\to 0} \frac{\log \nu(B(x,r))}{\log r}.
\end{equation*}
The local dimension looks at the measure of a ball around a fixed point $x$, while the Minkowski dimension always takes the point $y_r$ at scale $r$ for which $\nu(B(y_r,r))$ is minimal, the `least accessible part' of the attractor. For self-similar sets taking a point $x^*$ which maximises $\dim_\mathrm{loc}(\nu_{\p},x)$ automatically provides a sequence $y_r\equiv x^*$ that gives $\dim_{\mathrm{M}}\nu_{\p}$. However, for the Bedford--McMullen carpets we will show that $\dim_{\mathrm{M}}\nu_{\p}$ can be strictly larger than $\max_x \dim_\mathrm{loc}(\nu_{\p},x)$, see Remark~\ref{remark:dimloc}. Another related concept is the quantization problem for probability measures, see~\cite{Kessebohmer_Quantization} for some background.

\section{Main results}\label{sec:results}

We begin with the asymptotic pointwise almost sure behaviour of $T_r(\ii,x_0)$. Let $N_r(X)$ denote the smallest number of open sets of diameter $r$ required to cover the set $X$. The \emph{lower dimension} of $X$ is
\begin{multline}\label{eq:dimL}
\dim_{\mathrm{L}} X=\sup \big\{\alpha: (\exists\, C>0) \text{ such that } (\forall\, 0<r<R \leq|X| \text{ and } x \in X)  \\
\big.N_{r}(B(x, R) \cap X) \ge C(R/r)^{\alpha}\big\}.
\end{multline}
Recall, $\rnu=\pi_*\rmu$ is the pushforward measure of $\rmu$ through the natural projection $\pi$ defined in~\eqref{eq:natproj}. If $\mu$ is a Bernoulli measure $\mu=\mu_{\p}$, then since the $\mu_\p$ measure of a cylinder set is independent of the order of digits, we clearly have $\overleftarrow{\mu_\p}=\mu_\p$. Thus, in case of the measure $\nu_\p=\pi_\ast \mu_\p$, which equivalently is the unique measure that satisfies
\begin{equation}\label{eq:invmeasure}
\nu_{\p}(\cdot) = \sum_{i=1}^N p_i \nu_{\p}(f_i^{-1}(\cdot)),
\end{equation}
one can replace $\rnu$ by $\nu_\p$.

\begin{thm}\label{thm:2}
Let $\mathcal{F}$ be an arbitrary IFS which satisfies~\eqref{eq:condition} and whose attractor $\Lambda$ satisfies $\dim_{\rm L}\Lambda>0$. Then for any measure $\mu$ on $\Sigma$ with one-sided exponential decay of correlations \eqref{eq:decay} and for $\mu$-a.e. $\ii$ and every $x_0\in \Lambda$
	\begin{equation*}
	\liminf_{r \rightarrow 0}\frac{\log T_r(\ii,x_0)}{-\log r} = \underline{\dim}_{\mathrm{M}}(\rnu) \;\text{ and }\; \limsup_{r \rightarrow 0}\frac{\log T_r(\ii,x_0) }{-\log r} = \overline{\dim}_{\mathrm{M}}(\rnu).
	\end{equation*}
In particular, if $\mu=\mu_{\p}$ is Bernoulli, then one can replace $\rnu$ by $\nu_\p$.
\end{thm}

Theorem \ref{thm:2} can easily be flipped over to obtain the decay rate of $d_{\mathrm H} \big(\mathcal{O}_n(\ii,x_0),\Lambda\big)$. 

\begin{cor}\label{thm:1}
Let $\mathcal{F}$ be an arbitrary IFS which satisfies~\eqref{eq:condition} and whose attractor $\Lambda$ satisfies $\dim_{\rm L}\Lambda>0$. Then for any measure $\mu$ with one-sided exponential decay of correlations \eqref{eq:decay} and for $\mu$-a.e. $\ii$ and every $x_0\in \Lambda$
	\begin{equation*}
	\liminf_{n \rightarrow \infty}\frac{\log d_{\mathrm H} \big(\mathcal{O}_n(\ii,x_0),\Lambda\big) }{-\log n} = \frac{1}{\overline{\dim}_{\mathrm{M}}(\rnu)} \;\text{ and }\;
	\limsup_{n \rightarrow \infty}\frac{\log d_{\mathrm H} \big(\mathcal{O}_n(\ii,x_0),\Lambda\big) }{-\log n} = \frac{1}{\underline{\dim}_{\mathrm{M}}(\rnu)}.
	\end{equation*}
Again, if $\mu=\mu_{\p}$ is Bernoulli, then one can replace $\rnu$ by $\nu_\p$.
\end{cor}

\begin{proof}
	From the definition of $T_r(\ii,x_0)$ it follows that $r\leq d_{\mathrm H} \big(\mathcal{O}_n(\ii,x_0),\Lambda\big)$ for every $n<T_{r}(\ii,x_0)$, in particular, we have $n\leq T_{ d_{\mathrm H} (\mathcal{O}_n(\ii,x_0),\Lambda)/2}(\ii,x_0)$. This and Theorem~\ref{thm:2} imply that
	\[
	\begin{split}
	\frac{1}{\overline{\dim}_{\mathrm{M}}(\rnu)}&=\liminf_{r \rightarrow 0}\frac{-\log r}{\log T_r(\ii,x_0)}\geq\liminf_{r \rightarrow 0}\frac{-\log d_{\mathrm H} \big(\mathcal{O}_{T_r(\ii,x_0)-1}(\ii,x_0),\Lambda\big)}{\log T_r(\ii,x_0)}\\
	&\geq\liminf_{n \rightarrow \infty}\frac{-\log d_{\mathrm H} \big(\mathcal{O}_{n-1}(\ii,x_0),\Lambda\big)}{\log n}=\liminf_{n \rightarrow \infty}\frac{-\log d_{\mathrm H} \big(\mathcal{O}_{n}(\ii,x_0),\Lambda\big)}{\log n}\\
	&\geq\liminf_{n \rightarrow \infty}\frac{-\log d_{\mathrm H} \big(\mathcal{O}_n(\ii,x_0),\Lambda\big) }{\log T_{d_{\mathrm H} (\mathcal{O}_n(\ii,x_0),\Lambda)/2}(\ii,x_0)}\geq\liminf_{r\to0}\frac{-\log r}{\log T_{r/2}(\ii,x_0)}=\frac{1}{\overline{\dim}_{\mathrm{M}}(\rnu)}.
	\end{split}\]
	The proof of the second part is similar.
\end{proof}

\begin{rem}\label{rem:0}
The condition of one-sided exponential decay seems necessary. The other condition $\dim_{\rm L}\Lambda>0$ seems purely technical and we believe the result should still hold without it. In particular, we will show in Section~\ref{sec:poslower} some examples when this condition can be omitted. It is assumed only to ensure that there are sufficiently many small balls, where $\dim_{\rm M}\pi_*\rmu$ is attained and so the initial point $x_0$ cannot cause a strict drop in the approximation by beginning the chaos game at the least accessible part of $\Lambda$. Note, however, that $\dim_{\rm L}\Lambda>0$ is not a very restrictive condition. For example, all self-affine sets are uniformly perfect~\cite{XieEtal_ProcAMS2003} which is equivalent to having positive lower dimension~\cite[Lemma 2.1]{KaenmakiEtal_UniformlyPerfect13}. As an application, we look at Bedford--McMullen carpets, see Section~\ref{sec:IntroBMCarpets}.
\end{rem}

Next we consider the expected value of the cover time $\E_\mu T_r(x_0)$, which denotes the expectation of $T_r(\i,x_0)$ with respect to the measure $\mu$. We find that, roughly speaking, this can be bounded in terms of $(1/r)^\alpha$ (the reciprocal of the measure of the ball of minimum measure at scale $r$), up to some logarithmic correction factors. Define
\begin{equation}\label{o}
\underline{o}(r):= \max_{x \in \Lambda} \frac{\log \rnu(B(x,r))}{\log r} - \underline{\alpha}\;\text{ and }\; \overline{o}(r):= \max_{x \in \Lambda} \frac{\log \rnu(B(x,r))}{\log r} - \overline{\alpha}.
\end{equation}

\begin{thm} \label{thm:expectedvalue}
Let $\mathcal{F}$ be an arbitrary IFS which satisfies~\eqref{eq:condition}. Fix a $\sigma$-invariant measure $\mu$ on $\Sigma$ and let $\underline{\alpha}=\underline{\dim}_{\mathrm{M}}(\nu)$ and $\overline{\alpha}=\overline{\dim}_{\mathrm{M}}(\nu)$.
\begin{enumerate}[(a)]
\item If $\mu$ has one-sided exponential decay, there exists a constant\footnote{ The constant $C_1$ will be made explicit in the proof of Lemma \ref{a:proof}.}  $C_1$ such that for all $x_0 \in \Lambda$ and $r>0$ such that $|\overline{o}(r/4)|<\overline{\alpha}/2$ and $(r/4)^{\overline{\alpha}/2}<1/2\kappa$,
$$\E_\mu T_r(x_0) \leq C_1\left(\log(4/r)\right)^2 (r/4)^{-\overline{\alpha}-\overline{o}(r/4)},$$
where $\kappa$ is the constant defined in~\eqref{eq:decay}.
\item
If $\dim_{\mathrm{L}} \Lambda>0$ then for all $r>0$ sufficiently small\footnote{The assumptions on the size of $r$ will be made explicit in Lemma \ref{b:proof}.}  and all $x_0 \in \Lambda$,
$$\E_\mu T_r(x_0) \geq R_r^{-\underline{\alpha}+\underline{o}(R_r)}$$
where $ C_2(r(\log(1/r))^{2/\dim_{\mathrm{L}} \Lambda}) \leq R_r\leq C_3(r(\log(1/r))^{2/\dim_{\mathrm{L}} \Lambda})$ for some uniform constants $C_2,C_3$. The dependence of $R_r$ on $r$ and other parameters will be made explicit in the proof of Lemma \ref{b:proof}.
\end{enumerate}
\end{thm}

Since $r^{-\underline{\alpha}+\underline{o}(r)}$ and $r^{-\overline{\alpha}-\overline{o}(r)}$ are lower and upper bounds respectively on the measure of the ball of minimum measure at scale $r$, the main terms in Theorem \ref{thm:expectedvalue}(a) and (b) are essentially analogous to the main theorem in \cite{JurgaMorris_ChaosGame_2020arxiv}. Moreoever, Theorem \ref{thm:expectedvalue}(a) is essentially analogous to the upper bound on the cover time for the chaos game induced by a Bernoulli measure on self-similar sets in \cite{JurgaMorris_ChaosGame_2020arxiv} (except we have gained an extra logarithmic factor $\log(1/r)$ due to the fact that the measure is not necessarily Bernoulli but only has exponential decay). However, the lower bound in Theorem \ref{thm:expectedvalue}(b) is essentially just the trivial bound which estimates the expected cover time from below by the expected time to hit the ball of minimum measure. In order to improve this (to be analogous to the lower bound in \cite{JurgaMorris_ChaosGame_2020arxiv} for instance), one would need some extra information on where the balls of minimum measure were located, in particular to guarantee that they are sufficiently far from each other, in a dynamical sense.

\subsection{Examples with optimal rate} \label{measures}

From an applied point of view, the chaos game provides an efficient algorithm to produce images  well-approximating the attractor of an IFS. The choice of the measure $\mu$ driving the chaos game influences the ``quality" of the image we obtain. For any $r>0$, consider a maximal $r$-packing of the attractor $\Lambda$, i.e. a collection of sets of diameter $r$ with disjoint interiors that cover $\Lambda$. The orbit $\mathcal{O}_n(\ii,x_0)$ becomes $r$-dense in $\Lambda$ once it has visited all elements of the $r$-packing. Therefore, the practicality of the algorithm depends on how easy is it to define the $r$-packing and to keep track which elements of the $r$-packing the orbit has visited.

\subsubsection{Self-similar sets with SSP} If the IFS satisfies the \emph{strong separation property}, i.e.\linebreak $f_i(\Lambda)\cap f_j(\Lambda)=\emptyset$ for every $i\neq j$, then the natural projection $\pi:\Sigma\to\Lambda$ defined in~\eqref{eq:natproj} is a bijection between the symbolic space and the attractor. Hence, there is a chance to define the packing in terms of finite length strings. When the $f_i$ are all similarity mappings with contraction ratios $0<\lambda_i<1$, then the symbolic $r$-packing consists of those cylinder sets $[i_1,\ldots, i_m]$ for which $\lambda_{i_1}\cdot \ldots \cdot \lambda_{i_m}\leq r < \lambda_{i_1}\cdot \ldots \cdot \lambda_{i_{m-1}}$. If $j$ is the next chosen index, then the chaos game transitions from the current state $[i_1,\ldots, i_m]$ to the unique cylinder of the packing containing the cylinder $[j,i_1,\ldots, i_m]$. As mentioned before, $\dim_{\mathrm{M}}\Lambda$ is achieved by a self-similar measure $\nu_{\mathbf{p}}$ if and only if $p_i=\lambda_i^s$, where $s$, often called the \emph{similarity dimension}, is the unique solution to the equation $\sum_i\lambda_i^s=1$. By Theorem~\ref{thm:2} and~\eqref{eq:100}, the shortest running time to approximate $\Lambda$ with resolution $r$ is to choose $\mathbf{p}$ this way. Note that this also follows from~\cite{JurgaMorris_ChaosGame_2020arxiv}.

\subsubsection{Self-similar sets with overlaps} The separation condition is important, but can be circumvented in some cases in order to achieve the optimal possible rate. Suppose that the similarity dimension $s$ of the self-similar IFS $\{f_i(x)=\lambda_iO_ix+t_i\}_{i=1}^N$ on $\R^d$ is smaller than the dimension of the state space. Moreover, suppose that $\dim_{\rm M}\Lambda=s$, which happens generically, see \cite{Hochman_Annals14,HochmanIndDim}. Then for the self-similar measure $\nu_{\mathbf{p}}$ with $p_i=\lambda_i^s$ we have 
$$
\frac{\log\nu_{\mathbf{p}}((B(x,r))}{\log r}\leq\frac{\log\nu_{\mathbf{p}}(\pi([i_1,\ldots,i_m]))}{\log r}\leq s+\frac{s\min_i\log\lambda_i}{\log r},
$$
and so $\overline{\dim}_{\rm M}\nu_{\mathbf{p}}\leq s$. On the other hand, by our assumption $\underline{\dim}_{\rm M}\nu_{\mathbf{p}}\geq\underline{\dim}_{\rm M}\Lambda=s$.

\subsubsection{Bernoulli convolution} We saw that the most efficient convergence rate can be found for a wide class of self-similar systems if the similarity dimension is smaller than the dimension of the state space. However, this is not the case for Bernoulli convolutions. For $\lambda\in(1/2,1)$, consider the overlapping IFS
\begin{equation}\label{eq:BernoulliConv}
	\mathcal{F}=\{\lambda x-1,\, \lambda x+1\}, \quad\text{with attractor } \Lambda=\Big[ \frac{-1}{1-\lambda},\frac{1}{1-\lambda} \Big].
\end{equation}
The self-similar measure $\nu_{\mathbf{p}}$ associated to $\mathcal{F}$ is the well-known biased \emph{Bernoulli convolution}. There is extensive literature on it, especially for the non-biased $\mathbf{p}=(1/2,1/2)$ case, see the survey~\cite{PSS_60YearsBernoulli} or the recent influential papers~\cite{Shmerkin_Annals19,Varju_JAMS19} and references therein. Clearly, $\dim_{\rm M}\Lambda =s=1$, however it is easy to see that 
\begin{equation*}
	\dim_{\mathrm{M}}\nu_{\mathbf{p}}\geq\lim_{r\to 0}\max_{x\in\{-(1-\lambda)^{-1},(1-\lambda)^{-1}\}\}} \frac{\log \nu_{\mathbf{p}}(B(x,r))}{\log r}=\frac{\min\{\log p_1,\log p_2\}}{\log\lambda} \geq \frac{\log 2}{\log(1/\lambda)}>1
\end{equation*}
implying that there is no Bernoulli measure which achieves $\dim_{\rm M}\Lambda$. 

There are two further natural classes of measures satisfying one-sided exponential decay \eqref{eq:decay} which motivate the study of the chaos game beyond the i.i.d. and self-similar case. 

\subsubsection{Self-conformal sets} For conformal IFS $\{f_1,\ldots,f_N\}$ on $\R^d$, where the contractions are $C^{1+\epsilon}$, the natural measure $\rnu$ on the attractor $\Lambda$ can be expressed as the push-forward of the Gibbs measure $\rmu$ for the H\"older continuous potential $\ii\mapsto s\log|f_{i_1}'(\pi(\sigma\ii))|$, i.e. $\rnu=\pi_*\rmu$, where $s$ is the conformality dimension, see \cite[p. 89]{Falconer3-1997}. Such Gibbs measures with H\"older continuous potentials satisfy \eqref{eq:decay} (see \cite[p. 15]{bowen}),  thus our results are applicable to the reversed Gibbs measure $\mu$ (which clearly also satisfies \eqref{eq:decay}). In particular, there exists a constant $C>0$ such that for every $x\in\Lambda$ and every finite word $\iiv$ 
\begin{equation}
C^{-1}\leq\frac{\rmu([\iiv])}{|f_{\iiv}'(x)|^s}\leq C.
\end{equation}
Let us further assume that $\dim_{\mathrm{M}}\Lambda=s\leq d$, which holds for instance under the strong separation property (see \cite[p. 89]{Falconer3-1997}) or for typical systems (see \cite[Theorem~6.1]{SSU2}). Similarly to the self-similar case, it is easy to show in this case that $\dim_{\mathrm{M}}\rnu=s$. Hence, the chaos game driven by $\mu$ will distribute mass most uniformly over $\Lambda$.

\subsubsection{Self-affine sets with small dimension} Let $\{S_i(\cdot)=A_i(\cdot)+t_i\}_{i=1}^N$ be a self-affine IFS, where $A_i \in \mathcal{GL}(d,\R)$ and $t_i \in \R^d$.  The attractor of this IFS is called a \emph{self-affine} set $E$. The \emph{affinity dimension} $s_0$ is the expected Hausdorff and Minkowski dimension of $E$ \cite{falconer}. Generically, the measure of maximal Hausdorff dimension on $E$ is the pushforward $\rnu=\pi_*\rmu$ of a measure $\rmu$ on $\Sigma$ which is the equilibrium state for a particular sub-additive potential (which depends on $s_0$)  \cite{kaenmaki}, where by ``generically'' we mean in the sense that if a set of linear parts $A_i$ are fixed (where each $\norm{A_i}<1/2$) then the conclusion holds for Lebesgue typical choices of translations $t_i$, see \cite{JPS}. We shall refer to this measure $\rmu$ as a  \emph{K\"aenm\"aki measure}. Piraino proved that if the linear parts $\{A_i\}_{i=1}^N$ generate a strongly irreducible semigroup which contains a matrix with a simple leading eigenvalue, then \eqref{eq:decay} holds for the K\"aenm\"aki measure $\rmu$, see \cite[Theorem 3.3]{piraino}. Therefore, our results are applicable to the reversed  K\"aenm\"aki measure $\mu$, which is a natural candidate that we expect to optimise the cover time of the chaos game on $E$. If we additionally assume that $s_0 \leq 1$ then $\mu$ does indeed optimise the cover time of the chaos game, i.e. $\dim_{\mathrm{M}} \rnu=\dim_{\mathrm{M}} E=s_0$. To see this, note that under the assumptions of $\|A_i\|<1/2$, $\dim_{\mathrm{M}} E=s_0$ for Lebesgue almost every translations, see \cite{falconer_1988}, therefore $\dim_{\mathrm{M}} \rnu \geq s_0$. On the other hand, the strong irreducibility implies that $\rmu$ satisfies a type of ``Gibbs property'' for the weighted norm potential \cite[Remark~4.2]{KaeMor}, and using this and the definition of the Minkowski dimension, it is not difficult to show that $\dim_{\mathrm{M}} \rnu \leq s_0$. As a result, $\dim_{\mathrm{M}} \rnu=\dim_{\mathrm{M}} E=s_0$, i.e. $\mu$ optimises the cover time. Since $\dim_{\rm M} E=s_0$ under the assumptions of strong irreducibility and strong open set condition for planar systems, see \cite{bhr}, we can repeat the argument above for that situation as well.

Leaving the self-similar setting, except for special cases, the construction of the $r$-packing and keeping track of the orbit is difficult. One such special setting is the family of planar self-affine carpets. We continue by introducing the class of Bedford--McMullen carpets, give a complete characterisation of the vectors $\mathbf{p}$ which minimise $\dim_{\rm M}\nu_{\mathbf{p}}$ in Theorem~\ref{thm:3}, and later in Section~\ref{sec:simulations} present how to keep track of the orbit on the appropriately defined symbolic $r$-packing together with some simulations.


\subsection{Application to Bedford--McMullen carpets}\label{sec:IntroBMCarpets}

Bedford--McMullen carpets are self-affine sets on the plane introduced independently by Bedford~\cite{Bedford84_phd} and McMullen~\cite{mcmullen84}. There is an abundant amount of literature on them, we refer to the recent survey~\cite{fraser_BMcarpetSurvey_20arxiv} and references therein, thanks to the simplicity of their construction and at the same time exhibiting many interesting features. Here we show another such interesting feature regarding the optimisation problem
\begin{equation}\label{eq:01}
	\min_{\mathbf{p}} \dim_{\mathrm{M}}\nu_{\p}.
\end{equation}
As discussed in Remark~\ref{rem:0}, our results are applicable. In fact, an explicit formula is known for the lower dimension of Bedford--McMullen carpets~\cite{Fraser_TAMS2014}. Thus, any vector $\p^*$ minimising $\dim_{\mathrm{M}}\nu_{\p}$ has the interpretation that the chaos game run with $\p^*$ has the fastest running time among Bernoulli measures to reach a certain resolution.

Split $R=[0,1]^2$ into $m$ columns of equal width and $n$ rows of equal height for some integers $n>m\geq 2$ and consider orientation preserving maps on $R$ of the form
\begin{equation*}
	f_{(i,j)}(\underline{x}):= \begin{pmatrix} 1/m & 0 \\ 0 & 1/n \end{pmatrix} \begin{pmatrix} x \\ y \end{pmatrix} + \begin{pmatrix} i/m \\ j/n
	\end{pmatrix}
\end{equation*}
for the index set $(i,j)\in \mathcal{A}\subseteq \{0,\ldots,m-1\}\times\{0,\ldots,n-1\}$. The attractor $\Lambda$ of the IFS $\mathcal{F}=\{f_{(i,j)}\}_{(i,j)\in \mathcal{A}}$ is called a \emph{Bedford--McMullen carpet}, see Figure~\ref{fig:BMCarpets} for three examples. For our purposes it only matters how many maps there are in each column. Therefore, the input parameters of a carpet for us are the following.

Considering all non-empty columns, assume that the number of maps in a column take $M_0$ different values. In ascending order, let $N_1<N_2<\ldots<N_{M_0}$ denote these different values. Moreover, let $R_i$ denote the number of columns with $N_i$ number of maps. If $M_0=1$, then we say that the carpet has \emph{uniform vertical fibres}. The total number of non-empty columns is $M=R_1+\ldots+R_{M_0}\leq m$ and the total number of maps is $N=R_1N_1+\ldots+R_{M_0}N_{M_0}\leq nm$. For a distinguished index $1\leq K\leq M_0$, let $|\mathcal{R}_K|:= R_1+\ldots+R_K$, i.e. the number of columns with at most $N_K$ rectangles, and $\norm{\mathcal{R}^C_K}:=R_{K+1}N_{K+1}+\ldots+R_{M_0}N_{M_0}$, i.e. the total number of rectangles in columns with strictly more than $N_K$ rectangles.

\begin{figure}[h]
	\centering
	\includegraphics[width=0.97\textwidth]{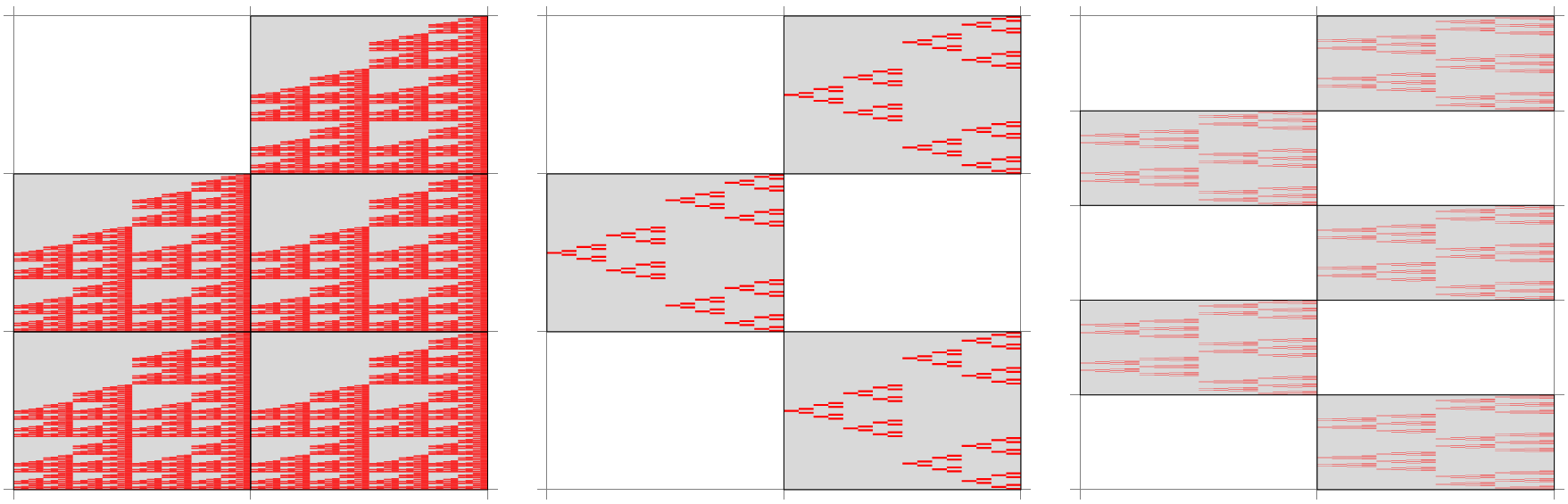}
	\caption{Three Bedford--McMullen carpets (in red) and the images of $[0,1]^2$ under the maps in each IFS (shaded rectangles). In each case, $\dim_{\mathrm{M}}\nu_{\mathbf{p}}$ is minimised by a different vector $\p$, see Table~\ref{table:1}.}
	\label{fig:BMCarpets}
\end{figure}

Recall, every non-degenerate probability vector $\p$ defines the measure $\nu_{\p}$ on $\Lambda$ via~\eqref{eq:invmeasure}. Let $\mathcal{Q}:=\big\{\q=(q_1,\dots,q_{M_0}):\, R_1q_1+\ldots+R_{M_0}q_{M_0}=1 \text{ and } q_k>0 \text{ for all }1\leq k \leq M_0\big\}$ and define the function
\begin{equation}\label{eq:08}
	\alpha(\mathbf{q}):=\max_{1\leq\, k,\ell\,\leq M_0} \left\{ \frac{\log (q_k/N_k)}{-\log n} + \left( 1-\frac{\log m}{\log n}\right) \frac{\log q_{\ell}}{-\log m} \right\}.
\end{equation}

\begin{thm}\label{thm:3}
Let $\Lambda$ be a Bedford--McMullen carpet with non-uniform vertical fibres. Then
\begin{equation*}
\min_{\mathbf{p}} \dim_{\mathrm{M}} \nu_{\mathbf{p}} = \min_{\mathbf{q}\in\mathcal{Q}} \alpha(\mathbf{q}).
\end{equation*}
Moreover, $\min_{\mathbf{q}\in\mathcal{Q}} \alpha(\mathbf{q})$ is attained at a vector either of the form $\mathbf{q}_K=(q_{K,1}, \ldots, q_{K,M_0})$ for a distinguished $1\leq K\leq M_0$ defined by
\begin{equation}\label{eq:04}
	q_{K,k}  =
	\begin{cases}
		N_K\cdot \big(N_K |\mathcal{R}_K|+\norm{\mathcal{R}^C_K}\big)^{-1},& \text{ for all } k\leq K, \\
		N_k\cdot \big(N_K |\mathcal{R}_K|+\norm{\mathcal{R}^C_K}\big)^{-1},& \text{ for all } k > K;
	\end{cases}
\end{equation}
or of the form $\mathbf{Q}_K=(Q_{K,1}, \ldots, Q_{K,M_0})$ for a $1\leq K\leq M_0-1$ defined by
\begin{equation}\label{eq:06}
	Q_{K,k}  = \left( 1-\frac{\log m}{\log n}\right)\frac{1}{|\mathcal{R}_K|}\text{ for all } k\leq K \;\;\text{ and }\;\; Q_{K,k}  = \frac{\log m}{\log n}\frac{N_k}{\norm{\mathcal{R}^C_K}}\text{ for all } k > K.
\end{equation}

Furthermore, if $\alpha(\q^*) = \min_{\mathbf{q}\in\mathcal{Q}} \alpha(\mathbf{q})$, then $\q^*$ defines a solution $\p^*$ to~\eqref{eq:01} by defining $p_i^*:= q^*_k/N_k$ if $i$ belongs to a column with $N_k$ rectangles (i.e. mass in each column is distributed uniformly amongst the rectangles within it).
\end{thm}

\begin{rem}\, \label{rem:1}
\begin{enumerate}
\item We give a procedure to determine which vector minimises $\alpha(\q)$ in Proposition~\ref{prop:opt}. For all three examples in Figure~\ref{fig:BMCarpets} a different vector is the minimiser, see Table~\ref{table:1}.
\item A Bedford--McMullen carpet $\Lambda$ has uniform vertical fibres if and only if its Hausdorff and Minkowski dimensions are equal. It is easy to see that in this case the vector minimising  $\dim_{\mathrm{M}}\nu_{\mathbf{p}}$ is the uniform vector $\p^*=(1/N,\ldots,1/N)$. Moreover, $\dim_{\mathrm{M}}\nu_{\mathbf{p}^*}=\dim_{\mathrm{M}}\Lambda$. If $\Lambda$ has non-uniform vertical fibres, then there is a positive gap  $\min_{\mathbf{p}}\dim_{\mathrm{M}}\nu_{\mathbf{p}}>\dim_{\mathrm{M}}\Lambda$. See Claim~\ref{claim:2} for details.
\item The vector minimizing $\dim_{\mathrm{M}}\nu_{\mathbf{p}}$ is not necessarily unique, see Section~\ref{sec:nonuniqueopt}.
\end{enumerate}
\end{rem}

\subsection*{Structure of paper}  In Section~\ref{sec:proofMain1} we prove Theorem~\ref{thm:2}, while Section~\ref{sec:proofExpected} contains the proof of Theorem~\ref{thm:expectedvalue}. The results for Bedford--McMullen carpets are proved in Section~\ref{sec:BMcarpet}. It also contains concrete examples and results of simulations. Section~\ref{sec:6} contains a short list of questions that arise naturally from our investigations.

\section{Proof of Theorem~\ref{thm:2}} \label{sec:proofMain1}

We begin by recalling and introducing necessary notation. Elements of the symbolic space $\Sigma=\{1,\ldots,N\}^{\N}$ are denoted $\ii,\jj$. The set of all finite words is $\Sigma^*=\bigcup_{n=0}^{\infty}\{1,\ldots,N\}^n$, where for $n=0$ we get the empty word and let $\overline{\Sigma}:=\Sigma \cup \Sigma^*$. Elements of $\Sigma^*$ are denoted by $\iiv,\jjv$ or as a truncation $\ii|n=i_1,\ldots i_n$ of an infinite word. For $\iiv=i_1 \ldots i_n$ we let $|\iiv|$ denote the length $n$ of the word $\iiv$. The left shift operator on $\overline{\Sigma}$ is $\sigma$, i.e. $\sigma(i_1i_2i_3\ldots) = i_2i_3\ldots$. The cylinder set defined by the finite word $\iiv$ of length $n$ is $[\iiv] = \{\ii\in\Sigma: \ii|n = \iiv\}$. 
The natural projection $\pi:\Sigma\to\Lambda$ is well defined by the limit
\begin{equation*}
\pi(\ii):=\lim_{n\to \infty} f_{\ii|n} (x_0),
\end{equation*}
where the limit is independent of the starting point $x_0$. The map $\pi$ is continuous, surjective, but may fail to be injective.

Let $\iiv_-$ denote the finite word which is obtained from $\iiv$ by dropping the last symbol of $\iiv$. Let
$$
\mathcal{P}_r:=\{\iiv\in\Sigma^*:|f_{\iiv}(\Lambda)|\leq r<|f_{\iiv_-}(\Lambda)|\}.
$$
For much of the time it is enough to work along subsequences of $\{2^{-n}\}_{n\in\N}$. Slightly abusing notation, we write $\mathcal{P}_n:=\mathcal{P}_{2^{-n}}$ and $T_{n}(\ii,x_0):=T_{2^{-n}}(\ii,x_0)$. This should make no confusion, since from the context it should be clear whether the subscript will tend to $\infty$ or $0$. Similarly, let $\mathcal{Q}_r(A)$ and $\mathcal{Q}_n(A)$ denote a maximal packing of $A$ by balls centered in $A$ with radius $r$ and $2^{-n}$ respectively. The definition of the Minkowski dimension implies that there exists $C_0>0$ such that
\begin{equation}\label{eq:boxup}
\#\mathcal{Q}_r(\Lambda)\leq C_0 r^{-2D} \;\;\text{ for every } n,
\end{equation}
where $D=\overline{\dim}_{\rm M}\Lambda$. In particular $\#\mathcal{Q}_n(\Lambda)\leq C_0 2^{2D n}$ for every n.

It readily follows from our assumption~\eqref{eq:condition} that for all $r>0$,
\begin{equation}\label{eq:length'}
L(r):= \frac{\log r}{\log a}-\frac{\log(|\Lambda|/a)}{\log a} \geq \max\{ |\iiv|:\, \iiv\in \mathcal{P}_{r}\}
\end{equation}
In particular for every $n\in\N$
\begin{equation}\label{eq:length}
L(n):= n\frac{\log2}{-\log a}-\frac{\log(|\Lambda|/a)}{\log a} \geq \max\{ |\iiv|:\, \iiv\in \mathcal{P}_{n}\}.
\end{equation}
For $\ii\in\Sigma$, let $\mathcal{P}_r(\ii)$ be the unique element $\iiv\in\mathcal{P}_r$ such that $\ii\in[\iiv]$. Let us define the symbolic ball as
$$
\widetilde{B}(x,r):=\{\iiv\in\mathcal{P}_r:B(x,r)\cap\pi[\iiv]\neq\emptyset\}.
$$
Then
$$
B(x,r)\cap \Lambda\subseteq\pi\widetilde{B}(x,r)\subseteq B(x,2r)\cap\Lambda.
$$
Recall the definition of lower dimension from~\eqref{eq:dimL}. If $d:=\dim_{\rm L}(\Lambda)>0$ then there exists a constant $c_0>0$ such that for every $0<r<R<|\Lambda|$ and $x\in \Lambda$
\begin{equation}\label{eq:lowerdim}
N_r(\Lambda\cap B(x,R))\geq c_0\left(\frac{R}{r}\right)^{d/2}.
\end{equation}

Theorem~\ref{thm:2} follows from the following two propositions.

\begin{prop}\label{prop:lower}
	Suppose that $\dim_{\rm L}\Lambda>0$ and let $\mu$ be an arbitrary $\sigma$-invariant measure. Let $n_k\to \infty$ be a sequence such that the limit
	$$
	\lim_{k\to\infty}\max_{y\in\Lambda}\frac{\log\rnu\big(B\big(y,\left((L(n_k)+2)/c_0\right)^{2/d}2^{-n_k+2}\big)\big)}{-n_k\log2}
	$$
	exists. Let $\alpha$ denote this limit. Then for $\mu$-a.e. $\ii$ and every $x_0\in \Lambda$
	$$
	\liminf_{k\to\infty}\frac{\log T_{n_k}(\ii,x_0)}{n_k\log 2}\geq\alpha.
	$$
\end{prop}

\begin{prop}\label{prop:upper}
	Let $\mu$ be a $\sigma$-invariant measure with exponential decay of correlation. Let $n_k\to\infty$ be a sequence such that the limit
	$$
	\lim_{k\to\infty}\max_{y\in\Lambda}\frac{\log\rnu(B(y,2^{-n_k}))}{-n_k\log2}
	$$
	exists. Let $\alpha$ denote this limit. Then for $\mu$-a.e. $\ii$ and every $x_0\in \Lambda$
	$$
	\limsup_{k\to\infty}\frac{\log T_{n_k}(\ii,x_0)}{n_k\log 2}\leq\alpha.
	$$
\end{prop}

\begin{proof}[Proof of Theorem~\ref{thm:2}]
  Let $n_k\to\infty$ be a sequence for which
	$$
	\underline{\dim}_{\rm M}(\rnu)=\lim_{k\to \infty} \max_{y\in\Lambda} \frac{\log \rnu(B(y,2^{-n_k}))}{-n_k\log2}.
	$$
	Then by Proposition~\ref{prop:upper} and Lemma~\ref{lem:equal}, for $\mu$-almost every $\ii$ and every $x_0\in\Lambda$
	$$
	\liminf_{n\to\infty}\frac{\log T_{n}(\ii,x_0)}{n\log 2}\leq\liminf_{k\to\infty}\frac{\log T_{n_k}(\ii,x_0)}{n_k\log 2}\leq\underline{\dim}_{\rm M}(\rnu).
	$$
	For an arbitrary $r>0$ there exists a unique integer $n$ such that $2^{-n}\leq r<2^{-(n-1)}$. Hence,
	$$
	\frac{n-1}{n}\frac{\log T_{n-1}(\ii,x_0)}{(n-1)\log 2}\leq\frac{\log T_r(\ii,x_0)}{-\log r}
	$$
	and so $\liminf\limits_{n\to\infty}\frac{1}{n\log 2}\log T_{n}(\ii,x_0)=\liminf\limits_{r\to0}\frac{1}{-\log r}\log T_r(\ii,x_0)$. This gives one direction. On the other hand
	$$
	\liminf_{r\to0}\max_{y\in\Lambda}\frac{\log\rnu(B(y,r))}{\log r}=\liminf_{n\to\infty}\max_{y\in\Lambda}\frac{\log\rnu\big(B\big(y,\left((L(n)+2)/c_0\right)^{2/d}2^{-n+2}\big)\big)}{-n\log2}.
	$$
	By Lemma~\ref{lem:equal}, this common value equals $\underline{\dim}_{\rm M}(\rnu)$. Hence, Proposition~\ref{prop:lower} implies that for $\mu$-a.e. $\ii$ and every $x_0\in\Lambda$
	$$
	\liminf_{n\to\infty}\frac{\log T_{n}(\ii,x_0)}{n\log 2}\geq\underline{\dim}_{\rm M}(\rnu),
	$$
proving the first assertion.	The proof of the other assertion is analogous.
\end{proof}
The next two subsections contain the proofs of the two remaining propositions. Both proofs are Borel-Cantelli arguments.

\subsection{Lower bound, proof of Proposition~\ref{prop:lower}}

In this proof we will estimate the cover time $T_n$ from below in terms of the time to hit the ``least accessible'' part of $\Lambda$ (i.e. the ball of minimum measure at scale $2^{-n}$). Moreover, we will use the positive lower dimension assumption to deduce that there are sufficiently many small balls of measure comparable to the ball of minimum measure, so that the choice of initial point for the chaos game cannot cause the cover time to drop by beginning in the least accessible part of $\Lambda$.

 Recall the notations $L(n)$ from~\eqref{eq:length} and $c_0$ from~\eqref{eq:lowerdim}. The first time an orbit hits a ball around a point $y$ with radius $r$ is denoted by
\begin{equation}\label{eq:tiyx}
\mathcal{T}_r(\ii,y,x_0):= \min\{n\geq 1: f_{\overleftarrow{\ii|n}}(x_0)\in B(y,r)\}.
\end{equation}
Observe that $\mathcal{T}_r(\ii,y,x_0)\leq T_r(\ii,x_0)$ for every $y\in\Lambda$. Let $n_k\to\infty$ be a sequence such that
\begin{equation*}
\alpha := \lim_{k\to\infty}\max_{y\in\Lambda}\frac{\log\rnu\big(B\big(y,\left((L(n_k)+2)/c_0\right)^{2/d}2^{-n_k+2}\big)\big)}{-n_k\log2},
\end{equation*}
moreover, let $y_k' \in \Lambda$ be a point for which the maximum is attained for $n_k$. Therefore, for every integer $K\geq1$, we can find $N(K)$  such that
	\begin{equation} \label{measure-ub}
	\rnu\big(B\big(y_k',\left((L(n_k)+2)/c_0\right)^{2/d}2^{-n_k+2}\big)\big) \leq 2^{-n_k(\alpha-\frac{1}{2K})} \quad\text{ for every } k\geq N(K).
	\end{equation}
	
Choosing $r=2^{-n_k+2}$ and $R=((L(n_k)+2)/c_0)^{2/d}\cdot r$ and letting $D$ denote the dimension of the state space $\R^D$ containing $\Lambda$, we see that
\begin{equation}\label{eq:onlower}
 L(n_k)+2 \leq N_{r}\big(\Lambda\cap B\big(y_k',R\big)\big)\leq \#\mathcal{Q}_{n_k-1}\big(\Lambda\cap B\big(y_k',R\big)\big) \leq \left(\frac{2R}{r}\right)^D=2^D((L(n_k)+2)/c_0)^{2D/d},
\end{equation}
where the upper bound is trivial and the lower bound follows from~\eqref{eq:lowerdim}. Hence, for each $\ii \in \Sigma$ there exists $y_k(\ii)\in\Lambda$ such that $B(y_k(\ii),2^{-n_k+1})\cap\mathcal{O}_{\left\lceil L(n_k)\right\rceil}(\ii,x_0)=\emptyset$ and
\begin{equation*}
B(y_k(\ii),2^{-n_k+1})\subset B\big(y_k',\left((L(n_k)+2)/c_0\right)^{2/d}2^{-n_k+2}\big).
\end{equation*}

	
	Let $\jj$ be an arbitrary coding of $x_0$. Then,
	
	\begin{align}
	&\mu\left(\left\{\ii:T_{2^{-n_k}}(\ii,x_0)<2^{n_k(\alpha-1/K)}\right\}\right)\nonumber\\
	&\qquad\leq\mu\left(\left\{\ii:\mathcal{T}_{2^{-n_k}}(\ii,y_k(\ii),x_0)<2^{n_k(\alpha-1/K)}\right\}\right)\nonumber\\
	&\qquad\leq\mu\left(\left\{\ii:\text{ there exists $ \lceil L(n_k)\rceil+1\leq\ell<2^{n_k(\alpha-1/K)}$ s.t. }f_{\overleftarrow{\ii|\ell}}(x_0)\in B(y_k(\ii),2^{-n_k})\right\}\right)\nonumber\\
	&\qquad\leq\mu\left(\left\{\ii:\text{ there exists $ \lceil L(n_k)\rceil +1\leq\ell<2^{n_k(\alpha-1/K)}$ s.t. }\mathcal{P}_{n_k}(\overleftarrow{\ii|\ell}\jj)\in\widetilde{B}(y_k(\ii),2^{-n_k})\right\}\right)\label{a1}
\end{align}
where the first inequality follows because for all $\ii \in \Sigma$, $T_{2^{-n_k}}(\ii,x_0)\geq\mathcal{T}_{2^{-n_k}}(\ii,y_k(\ii),x_0)$, the second inequality is because for all $\ii \in \Sigma$, $y_k(\ii)$ is chosen such that $B(y_k(\ii),2^{-n_k+1})\cap\mathcal{O}_{\left\lceil L(n_k)\right\rceil}(\ii,x_0)=\emptyset$ and the final inequality is because $B(y_k(\ii),2^{-n_k}) \subseteq \widetilde{B}(y_k(\ii),2^{-n_k})$.

Next, let $\mathcal{Y}_k$ denote the set of centres of the balls in $\mathcal{Q}_{n_k-1}\big(\Lambda\cap B\big(y_k',R\big)\big)$. The subset of $\Sigma$ that appears in \eqref{a1} is easily seen to be contained inside
$$ \bigcup_{y \in \mathcal{Y}_k} \bigcup_{\ell=\lceil L(n_k)\rceil+1}^{\lfloor 2^{n_k(\alpha-1/K)} \rfloor}\bigcup_{\iiv\in\widetilde{B}(y,2^{-n_k})}\left\{\ii: \mathcal{P}_{n_k}(\overleftarrow{\ii|\ell}\jj)=\iiv\right\}.$$
Thus we can bound \eqref{a1} by sum of measures of sets appearing in the above union to obtain that for $k \geq N(K)$

\begin{align*}
	\mu\left(\left\{\ii:T_{2^{-n_k}}(\ii,x_0)<2^{n_k(\alpha-1/K)}\right\}\right) &\leq\sum_{y\in\mathcal{Y}_k}\sum_{\ell=\lceil L(n_k)\rceil+1}^{\lfloor 2^{n_k(\alpha-1/K)}\rfloor}\sum_{\iiv\in\widetilde{B}(y,2^{-n_k})}\mu\left(\left\{\ii:\mathcal{P}_{n_k}(\overleftarrow{\ii|\ell}\jj)=[\iiv]\right\}\right) \\
	&\leq 2^{n_k(\alpha-1/K)} \sum_{y \in \mathcal{Y}_k}\sum_{\iiv\in\widetilde{B}(y,2^{-n_k})}\mu([\overleftarrow{\iiv}])\\
	&\leq 2^{n_k(\alpha-1/K)} \sum_{y \in \mathcal{Y}_k} \rnu(B(y,2^{-n_k+1}))\\
	&\leq2^D((L(n_k)+2)/c_0)^{2D/d}2^{n_k(\alpha-1/K)}2^{-n_k(\alpha-1/(2K))}\\
	&=2^D((L(n_k)+2)/c_0)^{2D/d}2^{-n_k/(2K)}. \nonumber
	\end{align*}
The second inequality follows because $\mathcal{P}_{n_k}(\overleftarrow{\ii|\ell}\jj)$ depends only on $\ii$ for $\ell \geq \lceil L(n_k)\rceil$ by \eqref{eq:length}. The third inequality follows because $\mu([\overleftarrow{\iiv}])=\rmu([\iiv])$ and $\pi \widetilde{B}(y,2^{-n_k}) \subseteq B(y,2^{-n_k+1})$. The fourth inequality follows by~\eqref{measure-ub},  \eqref{eq:onlower}, and the fact that for any $y \in \mathcal{Y}_k$, $B(y,2^{-n_k+1})\subset B\big(y_k',\left((L(n_k)+2)/c_0\right)^{2/d}2^{-n_k+2}\big)$.
Thus, the Borel-Cantelli lemma implies that
	$$
	\mu\left(\left\{\ii:T_{2^{-n_k}}(\ii,x_0)<2^{n_k(\alpha-1/K)}\text{ for infinitely many $k$'s}\right\}\right)=0.
	$$
	Since $K\geq1$ was arbitrary, we get
	$$
	\mu\left(\left\{\ii:\liminf_{k\to\infty}\frac{\log T_{2^{-n_k}}(\ii,x_0)}{n_k\log2}\geq\alpha\right\}\right)=1.
	$$

\subsection{Notes on positive lower dimension}\label{sec:poslower} Let us observe that the condition $\dim_L\Lambda>0$ is purely technical and used only to show the independence of the initial point $x_0\in\Lambda$ by providing sufficiently large collection of balls with approximately the smallest possible measure. However, this condition can be circumvented by some other conditions. For example, the lower dimension of the attractor of the system $\{f_1(x)=x^{1/x}/3,f_2(x)=(x+2)/3\}$ is $0$, see \cite[Section~6.3]{fraser_2020Book} but the claim of Proposition~\ref{prop:lower}, and in particular Theorem~\ref{thm:1}, holds for this system as well with measures with one-sided exponential decay.

\begin{prop}\label{prop:lower2}
Let $\mathcal{F}$ be an IFS satisfying~\eqref{eq:condition} and the strong separation property. Moreover, let us assume that there exist $\iiv\neq\jjv\in\Sigma^*$ and a constant $0<b<1$ such that 
$$
\|f_{\iiv}(x)-f_{\iiv}(y)\|,\|f_{\jjv}(x)-f_{\jjv}(y)\|>b\|x-y\|\text{ for all }x,y\in\Lambda.
$$
 Let $\mu$ be a $\sigma$-invariant measure such that there exists a constant $C>0$ such that for every $\iiv,\jjv\in\Sigma^*$
 \begin{equation}\label{eq:quasiBernoulli}
 \mu([\iiv\jjv])\leq C\mu([\iiv])\mu([\jjv]).
 \end{equation}
 Let $n_k\to \infty$ be a sequence such that the limit
	$$
	\lim_{k\to\infty}\max_{y\in\Lambda}\frac{\log\rnu\big(B\big(y,b^{-1}L(n_k)^{-\log b}2^{-n_k}\big)\big)}{-n_k\log2}
	$$
	exists. Let $\alpha$ denote this limit. Then for $\mu$-a.e. $\ii$ and every $x_0\in \Lambda$
	$$
	\liminf_{k\to\infty}\frac{\log T_{n_k}(\ii,x_0)}{n_k\log 2}\geq\alpha.
	$$
\end{prop}

\begin{proof}
Let $y_k'\in\Lambda$ be where $\max_{y\in\Lambda}\frac{\log\rnu\big(B\big(y,L(n_k)^{-\log b}2^{-n_k}\big)\big)}{-n_k\log2}$ is attained. By the strong separation, there exists $\delta>0$ such that $d(f_{\iiv}(\Lambda),f_{\jjv}(\Lambda))>\delta$. Let $\hbar_1,\ldots,\hbar_{2^{m_k}}\in\{\iiv,\jjv\}^{m_k}$ be words such that $\log L(n_k)\leq m_k<\log L(n_k)+1$. By the strong separation
$$
B(f_{\hbar_i}(y_k'),2^{-n_k})\cap B(f_{\hbar_j}(y_k'),2^{-n_k})=\emptyset
$$
for every $i\neq j$ and $k$ sufficiently large. Indeed, if the intersection above was non-empty then 
$$
2^{-n_k+1}\geq\|f_{\hbar_i}(y_k')-f_{\hbar_j}(y_k')\|\geq b^{m_k}\delta\geq b L(n_k)^{-\log b}\delta,
$$
which is absurd as $k\to\infty$. Hence, by the definition of $m_k$ for every $\ii\in\Sigma$ there exists $y_k(\ii)\in\{f_{\hbar_i}(y_k')\}_{i=1}^{2^{m_k}}$ such that $B(y_k(\ii),2^{-n_k})\cap\mathcal{O}_{\left\lceil L(n_k)\right\rceil}(\ii,x_0)=\emptyset$.

It is easy to see that $\rmu$ satisfies \eqref{eq:quasiBernoulli} too, and by using the strong separation, we get that for every $\rnu$-measurable set $E$ and finite word $\hbar\in\Sigma_*$
\begin{equation}\label{eq:technu}
\rnu(f_{\hbar}(E))=\rmu(\sigma^{-|\hbar|}\pi^{-1}(E)\cap[\hbar])\leq C\rmu(\sigma^{-|\hbar|}\pi^{-1}(E))\rmu([\hbar])\leq C\rmu(\pi^{-1}(E))=C\rnu(E).
\end{equation}

Moreover, by our assumptions on $f_{\iiv},f_{\jjv}$
$$
B(f_{\hbar_i}(y_k'),2^{-n_k})\subseteq f_{\hbar_i}\left(B(y_k',b^{-m}2^{-n_k})\right)\subseteq f_{\hbar_i}\left(B\big(y_k',b^{-1}L(n_k)^{-\log b}2^{-n_k}\big)\big)\right),
$$
for all $i=1,\ldots,2^m$. Hence, by \eqref{eq:technu}
$$
\rnu\left(f_{\hbar_i}\left(B\big(y_k',b^{-1}L(n_k)^{-\log b}2^{-n_k}\big)\big)\right)\right)\leq C\rnu\left(B\big(y_k',b^{-1}L(n_k)^{-\log b}2^{-n_k}\big)\big)\right)
$$
and so, for any $K\geq1$ and every sufficiently large $k$
$$
\frac{\log\rnu(B(f_{\hbar_i}(y_k'),2^{-n_k}))}{-n_k\log 2}\geq\frac{\log\rnu\left(B\big(y_k',b^{-1}L(n_k)^{-\log b}2^{-n_k}\big)\big)\right)}{-n_k\log2}+\frac{\log C}{-n_k\log2}\geq\alpha-\frac{1}{2K}.
$$
Now, the proof now can be finished as in Proposition~\ref{prop:lower} and left for the reader.
\end{proof}

Observe that the strong separation played a crucial role in the Proof of Proposition~\ref{prop:lower2}. Since in the overlapping case it might happen that \eqref{eq:technu} fails and the $\rnu$ measure of $f_{\hbar}(B(y,r))$ is much larger than the measure of the set $B(y,r)$.

\subsection{Upper bound, proof of Proposition~\ref{prop:upper}}

The upper bound is also a Borel-Cantelli argument, and for this we will need to estimate the probability that the cover time is asymptotically larger than what is claimed. In order to estimate this probability, we will use the one-sided exponential decay of correlations assumption to allow us to consider subsequent visits independently of one another.

Let $n_k \to \infty$ be a sequence such that
	$$
	\alpha:=\lim_{k\to\infty}\max_{y\in\Lambda}\frac{\log\rnu(B(y,2^{-n_k}))}{-n_k\log2}.
	$$
	Since
	$$
	\lim_{k\to\infty}\max_{y\in\Lambda}\frac{\log\rnu(B(y,2^{-(n_k+3)}))}{-n_k\log2}=\alpha
	$$
we can, for every $K \geq 1$, define $N(K)$ such that
\begin{equation} \label{measure-lb}
\min_{y\in\Lambda}\rnu(B(y,2^{-(n_k+3)}))\geq 2^{-n_k(\alpha+1/(2K))}
\end{equation}
for every $k\geq N(K)$.

	Let us consider the packing $\mathcal{Q}_{n_k+2}(\Lambda)$ of $\Lambda$ with balls of radius $2^{-n_k-2}$. Let $\mathcal{Y}_k$ denote the centres of balls in the packing $\mathcal{Q}_{n_k+2}$. Define
\begin{equation*}
t_{n_k}(\ii):= \min\big\{ m\geq 1:\, (\forall B\in\mathcal{Q}_{n_k+2})\,(\exists \ell\leq m) \text{ such that } f_{\overleftarrow{\ii|\ell}}(x_0)\in B\big\}.
\end{equation*}
First notice that for any $\ii \in \Sigma$, $T_{2^{-n_k}}(\ii,x_0)\leq t_{n_k}(\ii)$. Indeed, since $\{2B\}_{B\in\mathcal{Q}_{n_k+2}}$ is a cover of $\Lambda$, for every $x \in \Lambda$ there exists $B\in\mathcal{Q}_{n_k+2}$ such that $x\in 2B$ and, by definition of $t_{n_k}(\ii)$, there exists $1\leq\ell\leq t_{n_k}(\ii)$ such that $f_{\overleftarrow{\ii|\ell}}(x_0)\in B$. In particular $\|x-f_{\overleftarrow{\ii|\ell}}(x_0)\|\leq 2^{-n_k-1}+2^{-n_k-2}<2^{-n_k}$.

For short, let
$$m_k:=\big\lceil L(n_k+3)+\frac{(\alpha+1)}{\varepsilon}n_k \big\rceil=\big\lceil (n_k+3)\frac{\log2}{-\log a}-\frac{\log(|\Lambda|/a)}{\log a}+\frac{(\alpha+1)}{\varepsilon}n_k\big\rceil,$$
where $\varepsilon>0$ is defined in \eqref{eq:decay}.  Let $\jj$ be an arbitrary coding of $x_0$.


Then
\begin{align}
&\mu\left(\left\{\ii:T_{2^{-n_k}}(\ii,x_0)\geq \lceil 2^{(\alpha+1/K)n_k}\rceil m_k\right\}\right)\nonumber\\
&\quad\leq\mu\left(\left\{\ii:t_{n_k}(\ii,x_0)\geq \lceil 2^{(\alpha+1/K)n_k}\rceil m_k\right\}\right)\nonumber\\
&\quad\leq\mu\Big(\big\{\ii:\exists y\in\mathcal{Y}_k \text{ s.t. }  f_{\overleftarrow{\i|\ell m_k}}(x_0) \notin B(y, 2^{-n_k-2})\text{ for every $1\leq\ell\leq \lceil2^{(\alpha+1/K)n_k}\rceil$}\big\}\Big)\nonumber\\
&\quad\leq\mu\Big(\big\{\ii:\exists y\in\mathcal{Y}_k \text{ s.t. }\mathcal{P}_{2^{-n_k-3}}(\overleftarrow{\ii|\ell m_k}\jj)\notin \widetilde{B}(y, 2^{-n_k-3})\text{ for every $1\leq\ell\leq \lceil2^{(\alpha+1/K)n_k}\rceil$}\big\}\Big),\label{b1}
\end{align}
where the first inequality is because $T_{2^{-n}}(\ii,x_0)\leq t_{n}(\ii)$ and the final inequality follows because $\pi\widetilde{B}(y, 2^{-n_k-3}) \subseteq B(y, 2^{-n_k-2})$. Note that since $m_k\geq L(n_k+3)$, it follows that for any $\ii \in \Sigma$, $\mathcal{P}_{2^{-n_k-3}}(\overleftarrow{\ii| m_k} \jj)$ depends only on $\ii$.

Next, observe that by \eqref{eq:decay}, \eqref{eq:length} and the fact that $m_k-L(n_k+3) \geq \frac{\alpha+1}{\epsilon} n_k$ then for any $M \in \N$ and $\iiv_1, \ldots, \iiv_M \in \mathcal{P}_{n_k+3}$,
\begin{multline*}
\mu\left(\sigma^{-(m_k-|\iiv_1|)}([\overleftarrow{\iiv_1}]) \cap \sigma^{-(2m_k-|\iiv_2|)}([\overleftarrow{\iiv_2}]) \cap\cdots \cap \sigma^{-(Mm_k-|\iiv_M|)}([\overleftarrow{\iiv_M}]) \right) \\
 \leq (1+2^{-(\alpha+1)n_k}\kappa)^M\mu([\overleftarrow{\iiv_1}]) \cdots \mu([\overleftarrow{\iiv_M}]).
\end{multline*}

Applying this to \eqref{b1} we obtain that

\begin{align}
&\mu\left(\left\{\ii:T_{2^{-n_k}}(\ii,x_0)\geq \lceil 2^{(\alpha+1/K)n_k}\rceil m_k\right\}\right)\nonumber\\
&\qquad\leq (1+2^{-(\alpha+1)n_k}\kappa)^{\lceil 2^{(\alpha+1/K)n_k}\rceil}
\sum_{y \in \mathcal{Y}_k}\bigg(1-\sum_{\iiv\in\widetilde{B}(y, 2^{-n_k-3})}\mu([\overleftarrow{\iiv}])\bigg)^{\lceil 2^{(\alpha+1/K)n_k}\rceil}\nonumber\\
&\qquad\leq (1+2^{-(\alpha+1)n_k}\kappa)^{\lceil 2^{(\alpha+1/K)n_k}\rceil}\sum_{y \in \mathcal{Y}_k}\left(1-\rnu(B(y, 2^{-n_k-3}))\right)^{\lceil 2^{(\alpha+1/K)n_k}\rceil},\label{b2}
\end{align}
where the final inequality follows because $B(y, 2^{-n_k-3}) \subseteq \pi \widetilde{B}(y, 2^{-n_k-3})$.

Next, applying \eqref{measure-lb} to \eqref{b2} we obtain that for $k \geq N(K)$,

\begin{align}
&\mu\left(\left\{\ii:T_{2^{-n_k}}(\ii,x_0)\geq \lceil 2^{(\alpha+1/K)n_k}\rceil m_k\right\}\right)\nonumber\\
&\qquad\leq\#\mathcal{Q}_{n_k+2}(1+2^{-(\alpha+1)n_k}\kappa)^{\lceil 2^{(\alpha+1/K)n_k}\rceil}(1-2^{-n_k(\alpha+1/(2K))})^{\lceil 2^{(\alpha+1/K)n_k}\rceil}\nonumber\\
&\qquad\leq C_02^{2D(n_k+2)}\exp\left((2^{-(\alpha+1)n_k}\kappa-2^{-n_k(\alpha+1/(2K))})\lceil2^{(\alpha+1/K)n_k}\rceil\right)\nonumber\\
&\qquad \leq C_02^{2D(n_k+2)}\exp\left(-2^{n_k/(2K)}+2^{-(1-1/K)n_k}\kappa+2^{-(\alpha+1)n_k}\kappa-2^{-n_k(\alpha+1/2K)}\right),\nonumber
\end{align}
where in the second inequality we use \eqref{eq:boxup} and the fact that for all $x \in \R$, $1+x \leq \exp(x)$.

By the Borel-Cantelli lemma
$$
\mu\left(\left\{\ii:T_{2^{-n_k}}(\ii,x_0)\geq\lceil 2^{(\alpha+1/K)n_k}\rceil m_k\text{ for inf. many $k$'s}\right\}\right)=0.
$$
Since $K$ was arbitrary and $\lim_{k\to\infty}\frac{\log \lceil 2^{(\alpha+1/K)n_k}\rceil m_k}{n_k\log2}=\alpha+1/K$, we get
$$
\mu\left(\left\{\ii:\limsup_{k\to\infty}\frac{\log T_{2^{-n_k}}(\ii,x_0)}{n_k\log2}\leq\alpha\right\}\right)=1.
$$


\section{Expected cover time, proof of Theorem~\ref{thm:expectedvalue}} \label{sec:proofExpected}


Fix $x_0 \in \Lambda$ and let $\I$ be a finite set of pairwise disjoint Euclidean balls in $\R^d$. For $\i \in \Sigma$ we define $\mathcal{T}_{\I}(\i,x_0)$ to be the first time that the orbit $ \mathcal{O}_n(\i,x_0) $ has visited each of the balls in $\I$:
$$\mathcal{T}_{\I}(\i,x_0):=\inf\{n \geq 0: \forall I \in \I, \; \exists y \in \mathcal{O}_n(\i,x_0) \; \textnormal{such that} \; y \in I\}.$$
Similarly for $\i \in \Sigma$ and $I \in \I$ we define $\mathcal{T}_{I}(\i,x_0)$ to be the first time that the orbit $ \mathcal{O}_n(\i,x_0) $ visits the ball $I$:
$$\mathcal{T}_{I}(\i,x_0):=\inf\{n \geq 0: \exists y \in \mathcal{O}_n(\i,x_0) \; \textnormal{such that} \; y \in I\}.$$
 The expected values of $\mathcal{T}_{\I}(\i,x_0)$ and $\mathcal{T}_{I}(\i,x_0)$ with respect to a measure $\mu$ on $\Sigma$ are then denoted by $\E_\mu \mathcal{T}_{\I}(x_0)$ and $\E_\mu \mathcal{T}_I(x_0)$ respectively.


In the following proposition, we show that $\E_\mu \mathcal{T}_{\I}(x_0)$ can be bounded above and below in terms of uniform upper and lower bounds on $\E_\mu \mathcal{T}_{\I}(x)$ over $x \in \Lambda$. This is an adaptation of a method of Matthews \cite{matthews} for bounding the expected cover time of a Markov chain, meaning the expected time for the Markov chain to visit all of its states, in terms of the expected hitting times of individual states, see also \cite{peres}.

\begin{prop} \label{MM}
Fix $x_0 \in \Lambda$ and a finite set $\I$ of pairwise disjoint Euclidean balls in $\R^d$.

\begin{enumerate}[(a)]

\item Suppose there exists $C>1$ such that for all $\iiv,\jjv\in\Sigma^*$, $\mu([\iiv\jjv]) \leq C\mu([\iiv])\mu([\jjv])$. \nolinebreak \footnote{Note that this is clearly satisfied by any invariant measure satisfying the one-sided exponential decay property \eqref{eq:decay}.} Additionally assume that $ \sup_{I \in \I}\sup_{x \in \Lambda}\E_\mu \mathcal{T}_I(x) \leq T$. Then
$$\E_\mu\mathcal{T}_{\I}(x_0) \leq CT \left(1+\frac{1}{2}+ \cdots +\frac{1}{\# \I}\right).$$

\item Suppose there exists a constant $c>0$ such that $\mu([\iiv\jjv]) \geq c\mu([\iiv])\mu([\jjv])$ for all $\iiv,\jjv \in \Sigma^*$. Additionally assume that $\inf_{I \in \I}\inf_{x \in \Lambda \setminus I} \E_\mu \mathcal{T}_I(x) \geq t$.  Then
$$\E_\mu\mathcal{T}_{\I}(x_0) \geq ct \left(1+\frac{1}{2}+ \cdots +\frac{1}{\# \I}\right).$$
\end{enumerate}
\end{prop}

\begin{proof}
Write $N=\#\I$ and write $\I=\{I_1, \ldots, I_N\}$. Let $S_N$ denote the set of permutations $\sigma$ of $\{1, \ldots, N\}$. Let $m$ be the uniform measure on $S_N$, so that for all $\sigma \in S_N$, $m(\sigma)=\frac{1}{N!}$.

Given $\sigma \in S_N$, for $\i \in \Sigma$ we let $ \mathcal{T}_{\sigma(1)}^{\sigma(k)}(\i,x_0)$ denote the first time that the orbit $\mathcal{O}_n(\i,x_0)$ has visited each of the balls $I_{\sigma(1)}, \ldots, I_{\sigma(k)}$:
$$\mathcal{T}_{\sigma(1)}^{\sigma(k)}(\i,x_0):=\min\{n \geq 0: \forall 1 \leq j \leq k, \; \exists y \in \mathcal{O}_n(\i,x_0) \; \textnormal{such that} \; y \in I_{\sigma(j)}\}.$$
 The expected value of  $\mathcal{T}_{\sigma(1)}^{\sigma(k)}(\i,x_0)$ with respect to $\mu$ is denoted by $\E_\mu \mathcal{T}_{\sigma(1)}^{\sigma(k)}(x_0)$. For brevity, throughout this proof we will use the notation $\mathcal{T}_{\sigma(k)}$ to mean $\mathcal{T}_{I_{\sigma(k)}}$.

Given $\sigma \in S_N$ and $2 \leq k \leq N$ we define
$$A_{\sigma,k} :=\{\i \in \Sigma \; : \; \mathcal{T}_{\sigma(1)}^{\sigma(k-1)}(\i,x_0)<\mathcal{T}_{\sigma(k)}(\i,x_0)\}.$$
We also consider the following decomposition of $A_{\sigma,k}$ into cylinder sets. In particular there exists a set of words $C_{\sigma,k} \subset \Sigma^*$ such that $A_{\sigma,k}=\bigcup_{\iiv \in C_{\sigma,k}} [\iiv]$ and for each $\iiv \in C_{\sigma,k}$ and each $\i \in [\iiv]$, $\mathcal{T}_{\sigma(1)}^{\sigma(k-1)}(\i,x_0)=|\iiv|.$

 $\E_\mu \mathcal{T}_{\I}(x_0)=\E_\mu \mathcal{T}_{\sigma(1)}^{\sigma(N)}(x_0)$ for any $\sigma \in S_N$. Therefore,
\begin{eqnarray*}
\E_\mu \mathcal{T}_{\I}(x_0) &=& \int \E_\mu \mathcal{T}_{\sigma(1)}^{\sigma(N)}(x_0) dm \nonumber \\
&=& \int \E_\mu \mathcal{T}_{\sigma(1)}(x_0) dm + \sum_{k=2}^N \int \E_\mu(\mathcal{T}_{\sigma(1)}^{\sigma(k)}(x_0)-\mathcal{T}_{\sigma(1)}^{\sigma(k-1)}(x_0))dm \nonumber \\
&=& \int \E_\mu \mathcal{T}_{\sigma(1)}(x_0) dm + \sum_{k=2}^N \int \E_\mu(\mathcal{T}_{\sigma(1)}^{\sigma(k)}(\i,x_0)-\mathcal{T}_{\sigma(1)}^{\sigma(k-1)}(\i,x_0) \, : \, \i \in A_{\sigma,k})dm \\
&=&\int \E_\mu \mathcal{T}_{\sigma(1)}(x_0) dm + \sum_{k=2}^N \int \sum_{\iiv \in C_{\sigma,k}} \E_\mu(\mathcal{T}_{\sigma(1)}^{\sigma(k)}(\i,x_0)-\mathcal{T}_{\sigma(1)}^{\sigma(k-1)}(\i,x_0) \, : \, \i \in [\iiv])dm
\end{eqnarray*}
where the penultimate equality holds because if $\ii \notin A_{\sigma,k}$ then $\mathcal{T}_{\sigma(1)}^{\sigma(k)}(\ii,x_0)=\mathcal{T}_{\sigma(1)}^{\sigma(k-1)}(\ii,x_0)$.


For each $\iiv \in C_{\sigma,k}$ there exists $D_{\iiv} \subset \Sigma^*$ such that $[\iiv]=\bigcup_{\jjv \in D_{\iiv}}[\iiv\jjv]$ and for all $\i \in [\iiv\jjv]$, $\mathcal{T}_{\sigma(k)}(\i,x_0)=|\iiv|+|\jjv|$. In particular, $\E_\mu \mathcal{T}_{\sigma(k)}(f_{\overleftarrow{\iiv}}(x_0))=\sum_{\jjv \in D_{\iiv}} \mu([\jjv])|\jjv|$.

To prove (a), notice that since $\mu([\iiv\jjv]) \leq C\mu([\iiv])\mu([\jjv])$ for all $\iiv,\jjv \in \Sigma^*$,
\begin{multline*}
\sum_{\iiv \in C_{\sigma,k}} \E_\mu(\mathcal{T}_{\sigma(1)}^{\sigma(k)}(\i,x_0)-\mathcal{T}_{\sigma(1)}^{\sigma(k-1)}(\i,x_0) \, : \, \i \in [\iiv])=\sum_{\iiv \in C_{\sigma,k}} \sum_{\jjv \in D_{\iiv}} \mu([\iiv\jjv]) |\jjv| \\
\leq C\sum_{\iiv \in C_{\sigma,k}} \sum_{\jjv \in D_{\iiv}} \mu([\iiv])\mu([\jjv]) |\jjv|
= C\sum_{\iiv \in C_{\sigma,k}} \mu([\iiv]) \E_\mu \mathcal{T}_{\sigma(k)}(f_{\overleftarrow{\iiv}}(x_0))
\leq CT\mu(A_{\sigma,k}),
\end{multline*}
where in the final inequality we used that $ \sup_{I \in \I}\sup_{x \in \Lambda}\E_\mu \mathcal{T}_I(x) \leq T$. Therefore
\begin{equation*}
\E_\mu \mathcal{T}_{\I}(x_0)\leq CT \left(1+\sum_{k=2}^N \int \mu(A_{\sigma,k})dm\right).
\end{equation*}
To prove (b) notice that since $\mu([\iiv\jjv])\geq c\mu([\iiv])\mu([\jjv])$ for all $\iiv, \jjv \in \Sigma^*$, we similarly obtain
$$\sum_{\iiv \in C_{\sigma,k}} \E_\mu(\mathcal{T}_{\sigma(1)}^{\sigma(k)}(\i,x_0)-\mathcal{T}_{\sigma(1)}^{\sigma(k-1)}(\i,x_0) \, : \, \i \in [\iiv]) \geq c \sum_{\iiv \in C_{\sigma,k}} \mu([\iiv]) \E_\mu \mathcal{T}_{\sigma(k)}(f_{\overleftarrow{\iiv}}(x_0)) \geq ct \mu(A_{\sigma,k}), $$
where in the final inequality we used that $\inf_{I \in \I}\inf_{x \in \Lambda \setminus I} \E_\mu \mathcal{T}_I(x) \geq t$. Therefore
\begin{equation*}
\E_\mu \mathcal{T}_{\I}(x_0)\geq c t \left(1+\sum_{k=2}^N \int \mu(A_{\sigma,k})dm\right).
\end{equation*}

 Therefore to prove the proposition it suffices to show that $\int \mu(A_{\sigma,k})d m=\frac{1}{k}$. Fix $2 \leq k \leq N$. For each $\sigma \in S_N$ consider the unordered set $\{\sigma(1), \ldots, \sigma(k)\}$. Note that there are $\frac{N(N-1) \cdots (N-(k-1))}{k!}$ possible values that this set can take. For each possible value $\{i_1, \ldots, i_k\} \subset \{1, \ldots, N\}$ that this set can take, let $S_N(\{i_1, \ldots, i_k\})$ denote the set of all $\sigma$ for which $\{\sigma(1), \ldots, \sigma(k)\}=\{i_1, \ldots, i_k\}$, thinking of these as unordered sets.

Next, we can further separate each $S_N(\{i_1, \ldots, i_k\})$ into $k$ subsets $S_N^{i_j}(\{i_1, \ldots, i_k\})$, $(1 \leq j \leq k)$, which determines the set of all $\sigma \in S_N(\{i_1, \ldots, i_k\})$ for which $\sigma(k)=i_j$. Note that each $S_N^{i_j}(\{i_1, \ldots, i_k\})$ contains $(N-k)!(k-1)!$ permutations, corresponding to $(N-k)!$ ways to order the last $N-k$ terms and and $(k-1)!$ ways to arrange the first $k-1$ terms. Over each $\sigma \in S_N^{i_j}(\{i_1, \ldots, i_k\})$, the set $A_{\sigma,k}$ is constant. If for each $1 \leq j\leq k$ we choose a representative $\sigma_j \in S_N^{i_j}(\{i_1, \ldots, i_k\})$ then since the balls in $\I$ are pairwise disjoint, it follows that $\{A_{\sigma_j,k}\}_{j=1}^k$ are pairwise disjoint and $\bigcup_{j=1}^k A_{\sigma_j,k}=\Sigma$.

Hence for any choice of $\{i_1, \ldots, i_k\} \subset \{1, \ldots, N\}$,
\begin{eqnarray}
\int_{S_N(\{i_1, \ldots, i_k\})} \mu(A_{\sigma,k})d m=\sum_{n=1}^k \frac{(k-1)!(N-k)!}{N!} \mu(A_{\sigma_j,k}),
\label{part-integral}
\end{eqnarray}
where the factor $\frac{1}{N!}$ comes from the fact that $m$ is uniformly distributed. Now, since $\bigcup_{j=1}^k A_{\sigma_j,k}=\Sigma$ and  $\{A_{\sigma_j,k}\}_{j=1}^k$ are pairwise disjoint we have
$$\mu(A_{\sigma_k,k})= 1-\sum_{j=1}^{k-1} \mu(A_{\sigma_j,k})$$
and substituting this into \eqref{part-integral} we obtain
\begin{align*}
\int_{S_N(\{i_1, \ldots, i_k\})}\!\! \mu(A_{\sigma,k})d m&= \sum_{j=1}^{k-1}\frac{(k-1)!(N-k)!}{N!} \mu(A_{\sigma_j,k}) +\frac{(k-1)!(N-k)!}{N!} \Big(1-\sum_{j=1}^{k-1} \mu(A_{\sigma_j,k})\Big) \\
&=\frac{(k-1)!(N-k)!}{N!}.
\end{align*}
Therefore,
$$\int_{S_N} \mu(A_{\sigma,k}) dm=\frac{N(N-1) \cdots (N-(k-1))}{k!} \cdot \frac{(k-1)!(N-k)!}{N!}=\frac{1}{k}.$$
\end{proof}

We now establish the upper bound for $\E_\mu T_r(x_0)$ from Theorem \ref{thm:expectedvalue}(a). By Proposition \ref{MM} it is sufficient to estimate from above the expected hitting time to the ball $I$ of minimum measure at scale $r$. This can be estimated by bounding the probability of slow hitting times to $I$, and for this estimate the one-sided exponential decay of correlations assumption will be required to allow us to consider different segments of the orbit under the chaos game independently of each other.

\begin{lma}[Proof of Theorem \ref{thm:expectedvalue}(a)]\label{a:proof}

Suppose $\mu$ has one-sided exponential decay. Let $\overline{\alpha}=\overline{\dim}_{\mathrm{M}}(\rnu)$. There exists a constant $C_1$ (which will be made explicit) such that for all $x_0 \in \Lambda$ and $r>0$ such that $|\overline{o}(r/4)|<\overline{\alpha}/2$ and $(r/4)^{\overline{\alpha}/2}<1/2\kappa$,
$$\E_\mu T_r(x_0) \leq C_1\left(\log(4/r)\right)^2 (r/4)^{-\overline{\alpha}-\overline{o}(r/4)},$$
where $\kappa$ is the constant defined in~\eqref{eq:decay}.
\end{lma}

\begin{proof}
Recall from~\eqref{o} that $\overline{o}(r)$ was defined as
$$\overline{o}(r):= \max_{x \in \Lambda} \frac{\log \rnu(B(x,r))}{\log r} - \overline{\alpha}.$$
Fix $\epsilon>0$. For each $r>0$ for short, let $\I_r:=\mathcal{Q}_{r/4}(\Lambda)$ be a maximal centred packing of $\Lambda$ by disjoint balls of radius $\frac{r}{4}$. By \eqref{eq:boxup} $\log \# \I_r \leq  \log C_0+ 2D\log(4/r)$ for all $r>0$.  Also, it'll be useful to keep in mind that $\log(4/r)> \frac{2\log 2\kappa}{\overline{\alpha}}$ by our assumptions on $r$. Since the collection of balls of radius $\frac{r}{2}$ given by $\{2I\}_{I \in \I_r}$ forms a cover of $\Lambda$, it is easy to see that $\E_\mu T_r \leq \E_\mu \mathcal{T}_{\I_r}$. We will use Proposition \ref{MM} to bound $\E_\mu \mathcal{T}_{\I_r}$ from above. Therefore, we begin by showing that there exists a constant $C_1'$ such that for all $r$ such that $|\overline{o}(r/4)|<\overline{\alpha}/2$ and $(r/4)^{\overline{\alpha}/2}<1/2\kappa$,
\begin{equation*}
\sup_{I \in \I_r} \sup_{x \in \Lambda} \E_\mu \mathcal{T}_I(x) \leq \frac{C_1' \log(4/r)}{(r/4)^{\overline{\alpha}+\overline{o}(r/4)}}.\label{hit-ub}
\end{equation*}

 By definition of $\overline{o}(r)$,
$
\rnu(I)\geq \left(r/4\right)^{\overline{\alpha}+\overline{o}(r/4)}
$
for all $I \in \I_r$. Recall from \eqref{eq:length'} that
\begin{equation*}\label{eq:length2}
L(r):= \frac{\log r}{\log a}-\frac{\log(|\Lambda|/a)}{\log a} \geq \max\{ |\iiv|:\, \iiv\in \mathcal{P}_{r}\}
\end{equation*}
and define
$$
\ell(r):=L(r)+\frac{2\overline{\alpha}}{\varepsilon \log 2}\log(1/ r).
$$

Now, fix any $x \in \Lambda$ and let $\jj$ be an arbitrary coding of $x$. Then by one-sided exponential decay \eqref{eq:decay}
\begin{align*}
\mu\left(\left\{\ii:\mathcal{T}_I(x)>n\right\}\right)
&\leq\mu\left(\left\{\ii:\mathcal{P}_{r}(\overleftarrow{\ii|k \ell(r/4)}\jj)\notin I\text{ for every $1\leq k\leq \lceil n/\ell(r/4)\rceil$}\right\}\right)\nonumber\\
&\leq(1+2^{-\varepsilon(\ell(r/4)-L(r/4))}\kappa)^{\lceil n/\ell(r/4)\rceil}(1-\rnu(I))^{\lceil n/\ell(r/4)\rceil}\nonumber\\
&=(1+(r/4)^{2\overline{\alpha}}\kappa)^{\lceil n/\ell(r/4)\rceil}(1-\rnu(I))^{\lceil n/\ell(r/4)\rceil}\nonumber\\
&=(1+(r/4)^{2\overline{\alpha}}\kappa)^{\lceil n/\ell(r/4)\rceil}(1-(r/4)^{\overline{\alpha}+\overline{o}(r/4)})^{\lceil n/\ell(r/4)\rceil}\nonumber\\
&=(1-(r/4)^{\overline{\alpha}+\overline{o}(r/4)}+(r/4)^{2\overline{\alpha}}\kappa-(r/4)^{3\overline{\alpha}+o(r/4)}\kappa)^{\lceil n/\ell(r/4)\rceil}\nonumber\\
&\leq(1-(r/4)^{\overline{\alpha}+\overline{o}(r/4)}/2)^{n/\ell(r/4)},\nonumber
\end{align*}
where in the final inequality we have used that
$$(r/4)^{2\overline{\alpha}}\kappa-(r/4)^{3\overline{\alpha}+o(r/4)}\kappa<(r/4)^{2\overline{\alpha}}\kappa<(r/4)^{\overline{\alpha}+\overline{o}(r/4)}/2,$$
since $\overline{o}(r/4)<\overline{\alpha}/2$ and $(r/4)^{\overline{\alpha}/2}<1/2\kappa$. Hence,
\begin{align*}
\E_\mu \mathcal{T}_{I}(x)&=\sum_{n=0}^\infty\mu(\mathcal{T}_I(x)>n)\leq\sum_{n=0}^\infty (1-(r/4)^{\overline{\alpha}+\overline{o}(r/4)}/2)^{n/\ell(r/4)}\\
&=\frac{1}{1-(1-(r/4)^{\overline{\alpha}+\overline{o}(r/4)}/2)^{1/\ell(r/4)}}\leq\frac{2\ell(r/4)}{(r/4)^{\overline{\alpha}+\overline{o}(r/4)}},
\end{align*}
 where in the last inequality we used the Bernoulli inequality (Mitrinovic inequality), $(1+y)^\beta\leq1+\beta y$ for $0<\beta<1$ and $y\geq-1$, which is applicable since $\overline{o}(r/4)>-\overline{\alpha}$. Therefore we have proved \eqref{hit-ub}, where $C_1'=\frac{4\overline{\alpha}}{\varepsilon \log 2}-\frac{2}{\log a}-\overline{\alpha}\frac{\log|\Lambda|/a}{\log a \log 2\kappa}$, using that $\log(4/r)> \frac{2\log 2\kappa}{\overline{\alpha}}$ and the definition of $\ell(r/4)$.

Since $\mu$ satisfies the one-sided exponential decay property \eqref{eq:decay}, there exists $C>1$ such that for all $\iiv,\jjv\in\Sigma^*$, $\mu([\iiv\jjv]) \leq C\mu([\iiv])\mu([\jjv])$. Hence by Proposition \ref{MM}(a),
\begin{align*}
 \E_\mu T_r(x_0) \leq \E_\mu \mathcal{T}_{\I_r}(x_0) &\leq CC_1' \log(1/r)(r/4)^{-\overline{\alpha}-\overline{o}(r/4)} \left(1+ \cdots +\frac{1}{\#\I_r}\right) \\
&\leq  CC_1' \log(4/r)(r/4)^{-\overline{\alpha}-\overline{o}(r/4)}(1+\log  \# \I_r) \\
&\leq  CC_1' \log(4/r)(r/4)^{-\overline{\alpha}-\overline{o}(r/4)}(1+ \log C_0+ 2D\log(4/r)),
\end{align*}
 which completes the proof of the lemma by setting $C_1=CC_1' \big(2D+ \overline{\alpha}\frac{1+\log C_0}{2\log 2\kappa}\big)$, where again we have used that $\log(4/r)> \frac{2\log 2\kappa}{\overline{\alpha}}$.
\end{proof}

Next we establish the lower bound on $\E_\mu T_r(x_0)$ that appears in Theorem \ref{thm:expectedvalue}(b).
Recall $L(r)$ in \eqref{eq:length'} and $d=\dim_{\rm L}\Lambda>0$ in \eqref{eq:lowerdim}. Again by Proposition \ref{MM} it is sufficient to estimate from below the expected hitting time to the ball of minimum measure at scale $r$. We will use the assumption of positive lower dimension to ensure there are sufficiently many balls of measure comparable to that of the ball of minimum measure at scale $r$, so that starting the chaos game with an initial point $x_0$ lying in the least accessible part of $\Lambda$ will not cause a significant drop in the expectation $\E_\mu T_r(x_0)$.

\begin{lma}[Proof of Theorem \ref{thm:expectedvalue}(b)] \label{b:proof}
Assume $\dim_{\mathrm{L}} \Lambda >0$.  Denote $R_r:=2r\left(\frac{L(r)+2}{c_0}\right)^{2/d}.$ Then for all  $r>0$ such that $\min_{x \in \Lambda}\rnu(B(x,R_r))<\frac{1}{4}$ and all $x_0 \in \Lambda$
$$\E_\mu T_r(x_0) \geq  \frac{1}{4}R_r^{-\underline{\alpha}+\underline{o}(R_r)}.$$
\end{lma}

\begin{proof}
Recall that $\underline{o}(r)$ was defined as
$\underline{o}(r):= \max_{x \in \Lambda} \frac{\log \rnu(B(x,r))}{\log r} - \underline{\alpha}.$
Let $y'\in \Lambda$ be such that $\rnu\left(B\left(y',R_r\right)\right)$ is minimal and note that
\begin{equation}\rnu(B(y',R_r))=R_r^{-\underline{\alpha}-\underline{o}(R_r)}.\label{R_r} \end{equation}
Similarly to \eqref{eq:onlower},
$$
N_{2r}(\Lambda \cap B(y',R_r))=N_{2r}\left(\Lambda\cap B\left(y',2r\left(\frac{L(r)+2}{c_0}\right)^{2/d}\right)\right)\geq L(r)+2.
$$
Hence, $\#\mathcal{Q}_r(\Lambda\cap B\left(y',R_r\right))\geq L(r)+2$. So for every $\ii\in\Sigma$ there exists $y(\ii)\in\Lambda$ such that
\begin{equation*}
B(y(\ii),r)\cap\mathcal{O}_{\left\lceil L(n_k)\right\rceil}(\ii,x_0)=\emptyset\text{ and } B(y(\ii),r)\subset B\big(y',R_r\big).
\end{equation*}
Let $\mathcal{Y}_r$ be the set of centres of the balls in $\mathcal{Q}_r(\Lambda\cap B\left(y',R_r\right))$. Recall the definition \eqref{eq:tiyx} of $\mathcal{T}_r(\i,y,x_0)$. So
\begin{align*}
&\E_\mu(T_{r}(x_0))=\sum_{n=0}^\infty\mu(\i \in \Sigma:T_r(\i,x_0)>n)\geq\sum_{n=0}^\infty\mu(\i \in \Sigma:\mathcal{T}_r(\i,y(\ii),x_0)>n)\nonumber\\
&\qquad=\sum_{n=0}^\infty\left(1-\mu(\i \in \Sigma:\mathcal{T}_r(\i,y(\ii),x_0)\leq n)\right)=\sum_{n=0}^\infty\bigg(1-\sum_{y\in\mathcal{Y}_r}\mu(\i \in [\iiv]:\mathcal{T}_r(\i,y,x_0)\leq n)\bigg)\nonumber\\
&\qquad\geq \sum_{n=0}^\infty\max\{0,1-n\sum_{y\in\mathcal{Y}_r}\rnu(B(y,r))\}\geq\sum_{n=0}^\infty\max\left\{0,1-n\rnu\left(B\left(y',R_r\right)\right)\right\}\nonumber\\
&\qquad=\sum_{n=0}^{\left\lceil\rnu\left(B\left(y',R_r\right)\right)^{-1}\right\rceil}\max\left\{0,1-n\rnu\left(B\left(y',R_r\right)\right)\right\}\nonumber\\
&\qquad\geq\rnu\left(B\left(y',R_r\right)\right)^{-1}\left(\frac{1}{2}-\rnu\left(B\left(y',R_r\right)\right)\right),
\end{align*}
which completes the proof by \eqref{R_r} and since $\min_{x \in \Lambda}\rnu(B(x,R_r))<\frac{1}{4}$.
\end{proof}


\section{Bedford--McMullen carpets} \label{sec:BMcarpet}

In this section, we give an explicit procedure to determine a vector solving the optimisation problem~\eqref{eq:01} in Proposition~\ref{prop:opt}, prove the claims of Theorem~\ref{thm:3} and provide some additional insight through examples. Recall all the notation introduced in Section~\ref{sec:IntroBMCarpets}, in particular, the function $\alpha(\q)$ and vectors $\q_K, \mathbf{Q}_K$ from~\eqref{eq:08}, \eqref{eq:04}, and \eqref{eq:06}, respectively.

To each vector $\q_K$ and $\mathbf{Q}_K$ we associate the vector $\p_K=(p_{K,1},\ldots,p_{K,N})$ and $\mathbf{P}_K=(P_{K,1},\ldots,P_{K,N})$ by distributing mass within columns uniformly, i.e. we set $p_{K,i}=q_{K,k}/N_k$ if $i$ belongs to a column with $N_k$ rectangles and similarly $P_{K,i}=Q_{K,k}/N_k$. According to Theorem~\ref{thm:3}, one of these vectors solves the optimisation problem~\eqref{eq:01}.

Observe that $\mathbf{p}_1$ gives the \emph{uniform measure} $\mathbf{p}_1=(1/N,\ldots,1/N)$ and $\mathbf{p}_{M_0}$ gives the \emph{coordinate uniform measure} $\mathbf{p}_{M_0}=(1/(MN_{\psi(1)}),\ldots,1/(MN_{\psi(N)}))$, where the function $\psi: \{1,2,\ldots,N\}\to\{1,2,\ldots,M_0\}$ is defined
\begin{equation*}
\psi(i):=k,\;\; \text{ if $i$ belongs to a column with $N_k$ rectangles}.
\end{equation*}
Moreover, the interpretation of the vectors $\mathbf{P}_K$ is that all columns with at most $N_K$ maps are given mass $1-\frac{\log m}{\log n}$ and this mass is further distributed within these columns in a coordinate uniform way. While the remaining $\frac{\log m}{\log n}$ weight is given to columns with more than $N_K$ rectangles and this mass is distributed uniformly between all rectangles in these columns.

The uniform measure is the `natural measure' in the uniform vertical fibre case (recall, when all non-empty columns have the same number of rectangles). In the non-uniform vertical fibre case the coordinate uniform measure has the property that its lower and Assouad dimensions simultaneously realise the lower and Assouad dimension of the attractor provided the `\emph{very strong separation condition}' holds~\cite{FraserHowroyd_AnnAcadSciFennMath17}, see~\cite[Section 8.6]{fraser_2020Book} for additional information and definitions. In case of the Minkowski dimension there does not exist a self-affine measure $\nu_{\mathbf{p}}$ for which $\dim_{\mathrm{M}}\nu_{\mathbf{p}} = \dim_{\mathrm{M}}\Lambda$ (unless $\Lambda$ has uniform vertical fibres).

\begin{claim}\label{claim:2}
For a vector $\p=(p_1,\ldots,p_N)$, let $\q_{\p}=(q_{\p,1},\dots,q_{\p,M})$ denote the vector defined as
\begin{equation*}
q_{\p,\jh} := \text{sum of probabilities $p_i$ in the $\jh$-th column}.
\end{equation*}
Then, for any self-affine measure $\nu_{\p}$ on a Bedford--McMullen carpet
	\begin{equation}\label{eq:02}
		\dim_{\mathrm{M}}\nu_{\p} = \max_{\substack{1\leq\, i\,\leq N \\ 1\leq\, \jh\, \leq M}} \left\{ \frac{\log p_i}{-\log n} + \left( 1-\frac{\log m}{\log n}\right) \frac{\log q_{\p,\jh}}{-\log m} \right\}.
	\end{equation}
	As a result, $\min_{\mathbf{p}} \dim_{\mathrm{M}}\nu_{\p} = \dim_{\mathrm{M}}\Lambda$ if and only if $\Lambda$ has uniform vertical fibres.
\end{claim}
Formula~\eqref{eq:02} is stated in~\cite[Theorem 8.6.2]{fraser_2020Book} in the case of the very strong separation condition. For convenience of the reader, we provide the short argument for any Bedford--McMullen carpet $\Lambda$.

\begin{proof}
For an index $i\in\{1,\ldots,N\}$ let $\phi(i)\in\{1,\ldots,M\}$ denote the index of the column to which the rectangle $f_i([0,1]^2)$ belongs to. Furthermore, for $\ii\in\Sigma=\{1,\ldots,N\}^\mathbb N$ define $\Phi(\ii) = \phi(i_1)\phi(i_2)\ldots$. Then $\Sigma$ endowed with the metric $d(\ii,\jj):= m^{-|\Phi(\ii)\wedge\Phi(\jj)|} + n^{-|\ii\wedge\jj|}$ is a complete metric space, where $\ii\wedge\jj$ denotes the longest common prefix of $\ii$ and $\jj$. A level $K$ ball according to this metric is
\begin{equation}\label{eq:approxsquare}
	B_K(\ii) = \big\{\jj\in\Sigma:\, |\ii\wedge\jj|\geq L(K) \text{ and } |\Phi(\ii)\wedge\Phi(\jj)|\geq K \big\},
\end{equation}
where $L(K)$ is the unique  integer such that $m^{-K}\leq n^{-L(K)} < m^{-(K-1)}$. We call $B_K(\ii)$ a \emph{symbolic approximate square} at level $K$. Let $\mathcal{B}_K$ denote the set of level $K$ approximate squares. Each $B_K(\ii)$ can be identified with the sequence $(i_1,\ldots i_{L(K)},\phi(i_{L(K)+1}),\ldots,\phi(i_K))$.

As before, let $\mu_{\mathbf{p}}=\mathbf{p}^{\mathbb{N}}$. The $\mu_{\p}$ measure of an approximate square $B_K(\ii)$ is equal to
\begin{equation*}
	\mu_{\mathbf{p}}(B_K(\ii))= \prod_{\ell=1}^{L(K)} p_{i_{\ell}} \cdot \prod_{\ell=L(K)+1}^K q_{\p,\phi(i_{\ell})}.
\end{equation*}
Let $p^*=\min_i p_i$ and $q^*=\min_{\jh} q_{\p,\jh}$. Then
\begin{equation*}
	\min \big\{ \mu_{\mathbf{p}}(B_K(\ii)):\, B_K(\ii)\in \mathcal{B}_K \big\} = (p^*)^{L(K)} \cdot (q^*)^{K-L(K)} = \big(m^{-K}\big)^{\frac{\log p^*}{-\log n} + \left( 1-\frac{\log m}{\log n}\right) \frac{\log q^*}{-\log m}}.
\end{equation*}
This immediately implies that $\dim_{\mathrm{M}}\mu_{\p}$ equals the formula in~\eqref{eq:02}.

The next step is to show that $\dim_{\mathrm{M}}\nu_{\p}=\dim_{\mathrm{M}}\mu_{\p}$. The way that $d(\ii,\jj)$ is defined implies that up some uniform multiplicative constant $d(\ii,\jj)\approx \norm{\pi(\ii)-\pi(\jj)}$, where $\pi:\Sigma\to\Lambda$ is the natural projection defined in~\eqref{eq:natproj}. Hence, for any approximate square $\mathrm{diam}\big(\pi(B_K(\ii))\big)\approx m^{-K} \approx n^{-L(K)}$. Since $\pi$ can only increase the measure of a ball, we get $\dim_{\mathrm{M}} \nu_{\mathbf{p}}\leq \dim_{\mathrm{M}} \mu_{\mathbf{p}}$.

To see the other inequality, consider the approximate square $\pi(B_K(\ii))\subset \Lambda$ whose $\mu_{\mathbf{p}}$ measure is minimal. By assumption, on the first level at least two columns are non-empty and at least one of them has at least two maps. Thus, there is an $x\in\pi(B_K(\ii))$ and a constant $c$ independent of $K$ such that $B(x,c n^{-(L(K)+2)})\cap \Lambda\subset \pi(B_K(\ii))$. As a result,
\begin{equation*}
\nu_{\p}\big( B(x,c n^{-(L(K)+2)}) \big) \leq  \mu_{\mathbf{p}}(B_K(\ii)) \approx \big( c n^{-(L(K)+2)} \big)^{\dim_{\mathrm{M}} \mu_{\mathbf{p}}}.
\end{equation*}
This implies that $\dim_{\mathrm{M}} \nu_{\mathbf{p}}\geq \dim_{\mathrm{M}} \mu_{\mathbf{p}}$. Thus, $\dim_{\mathrm{M}} \nu_{\mathbf{p}} = \dim_{\mathrm{M}} \mu_{\mathbf{p}}$.

Finally, $\dim_{\mathrm{M}}\Lambda = \frac{\log N}{\log n} + \big( 1-\frac{\log m}{\log n}\big) \frac{\log M}{\log m}$. Hence, $\min_{\mathbf{p}}\dim_{\mathrm{M}}\nu_{\p} = \dim_{\mathrm{M}}\Lambda$ if and only if both $\mathbf{p}$ and $\mathbf{q}_{\mathbf{p}}$ are the uniform vectors on $\{1,\ldots,N\}$ and $\{1,\ldots,M\}$, respectively. This can happen only if each column has the same number of rectangles, i.e. $\Lambda$ has uniform vertical fibres.
\end{proof}

\begin{rem}\label{remark:dimloc}
The local dimension spectrum of Bedford--McMullen carpets was studied in~\cite{King_LocalDimBMCarpet_95AdvMath} and~\cite{JordanRams_MultifractaBMCarpet_2011}. In particular, they showed that the upper end of the spectrum equals
\begin{equation*}
\max_{x\in\Lambda} \dim_\mathrm{loc}(\nu_{\p},x) = \max_{1\leq\, i\,\leq N} \left\{ \frac{\log p_i}{-\log n} + \left( 1-\frac{\log m}{\log n}\right) \frac{\log q_{\p,\psi(i)}}{-\log m} \right\}.
\end{equation*}
This differs from~\eqref{eq:02} only in that the coordinate of $\q_{\p}$ can not be chosen independently, it has to be the coordinate corresponding to the column of $i$. Hence, $\max_{x\in\Lambda} \dim_\mathrm{loc}(\nu_{\p},x)\leq \dim_{\mathrm{M}}\nu_{\mathbf{p}}$ and there is a strict inequality if $p_i$ and $q_{\p,\jh}$ are minimised in different columns. This is the case for the vectors $\mathbf{P}_K$. Such a phenomena does not hold for self-similar sets.
\end{rem}

The next claim shows that we can reduce the minimisation problem $\min_{\mathbf{p}} \dim_{\mathrm{M}}\nu_{\p}$ to the lower dimensional problem $\min_{\mathbf{q}\in\mathcal{Q}} \alpha(\mathbf{q})$, where $\alpha(\q)$ was defined in~\eqref{eq:08} and $\mathcal{Q}=\big\{\q=(q_1,\dots,q_{M_0}):\, R_1q_1+\ldots+R_{M_0}q_{M_0}=1 \text{ and } q_k>0 \text{ for all }1\leq k \leq M_0\big\}$.

\begin{claim}\label{claim:1}
If $\mathbf{q}=(q_1,\ldots,q_{M_0})$ is such that $\alpha(\q) = \min_{\mathbf{q}'\in\mathcal{Q}} \alpha(\mathbf{q}')$, then $\q$ has the following two properties:
	\begin{enumerate}
		\item $\min_{\mathbf{p}} \dim_{\mathrm{M}}\nu_{\p} = \alpha(\mathbf{q})$, where $\mathbf{q}$ gives a solution $\mathbf{p}^*=(p_1^*, \ldots, p_N^*)$ to~\eqref{eq:01} by setting  $p^*_i:= q_k/N_k$ if $i$ belongs to a column with $N_k$ rectangles (i.e. mass is distributed uniformly within columns);
		\item There exists a unique $1 \leq K \leq M_0-1$ for which
		\begin{equation}\label{eq:05}
			q_1= q_2=\cdots=q_K \leq \min_{K+1 \leq k \leq M_0} q_k \;\;\; \textnormal{and} \;\;\; \frac{q_{K+1}}{N_{K+1}}=\frac{q_{K+2}}{N_{K+2}}= \cdots=\frac{q_{M_0}}{N_{M_0}} \leq \min_{1 \leq k \leq K}\frac{q_k}{N_k}.
		\end{equation}
	\end{enumerate}
\end{claim}

\begin{proof}
First observe that the maximization over the indices $k$ and $\ell$ in the definition~\eqref{eq:08} of $\alpha(\q)$ is independent. Moreover, $-\log q_k/N_k$ and $-\log q_{\ell}$ are maximal if and only if $q_k/N_k$ and $q_{\ell}$ are minimal, respectively.
	
Assertion $(1)$ simply follows from the fact that the value of $-\log \min_i p_i$ can only increase in the formula~\eqref{eq:02} for $\dim_{\mathrm{M}}\nu_{\mathbf{p}}$ if mass within each column is not distributed uniformly. This argument also implies that for each $k=1,\ldots,M_0$, the mass within the collection of $R_k$ columns should also be uniformly distributed amongst the columns. Thus, it is enough to consider the $M_0-1$ dimensional problem $\min_{\mathbf{q}\in\mathcal{Q}} \alpha(\mathbf{q})$ to solve $\min_{\mathbf{p}} \dim_{\mathrm{M}}\nu_{\p}$.
	
Assume $\mathbf{q}=(q_1,\ldots,q_{M_0})$ is such that $\alpha(\q) = \min_{\mathbf{q}'\in\mathcal{Q}} \alpha(\mathbf{q}')$. To see (2), first notice that for each $k=1,\ldots,M_0$ at least one of the following holds:
	\begin{equation*}
		(a)\;\, q_k = \min_{1\leq\,\ell\,\leq M_0} q_{\ell}, \qquad (b)\;\, q_k/N_k = \min_{1\leq\,\ell\,\leq M_0} q_{\ell}/N_{\ell}.
	\end{equation*}
	This is because if there existed an index $1\leq k_0\leq M_0$ such that $q_{k_0} > \min_{\ell} q_{\ell}$ and $q_{k_0}/N_{k_0} > \min_{\ell} q_{\ell}/N_{\ell}$, then mass could be transferred from that column into the columns attaining (either) minimum, therefore reducing the maximum in the definition of $\alpha(\q)$.
	
	Secondly, observe that if $1\leq K_1 \neq K_2 \leq M_0$ are two distinct indices such that $q_{K_1} = \min_{1\leq\,\ell\,\leq M_0} q_{\ell}$ and $q_{K_2}/N_{K_2} = \min_{1\leq\,\ell\,\leq M_0} q_{\ell}/N_{\ell}$, then $K_1<K_2$. Indeed by the choice of $K_1$ and $K_2$
	\begin{equation*}
		q_{K_2}/N_{K_2} \leq q_{K_1}/N_{K_1} \leq q_{K_2}/N_{K_1} \;\Longrightarrow\; N_{K_1} \leq N_{K_2}.
	\end{equation*}
	Since $N_k$ are in ascending order and $K_1\neq K_2$, it follows that $K_1<K_2$. In particular (2) holds.
	
	The uniqueness of $K$ follows by the simple observation that since $N_1<\cdots<N_{M_0}$ and $q_{K+i}/N_{K+i}=q_{K+i+1}/N_{K+i+1}$ for $i=1,\ldots,M_0-K-1$ we have $q_{K+i+1}>q_{K+i}$ for $i=1,\ldots,M_0-K-1$.
\end{proof}

\subsection{Finding a vector minimising \texorpdfstring{$\dim_{\mathrm{M}} \nu_{\mathbf{p}}$}{dimM nup}}

Claim~\ref{claim:1} shows that to find a vector that minimises $\dim_{\mathrm{M}} \nu_{\mathbf{p}}$, it is enough to find a vector that minimises $\alpha(\q)$ and then distribute mass within columns evenly amongst rectangles. The next proposition shows how to find a vector which minimises $\alpha(\q)$.

Recall, for an index $1\leq K\leq M_0$, we denote $|\mathcal{R}_K|= R_1+\ldots+R_K$ and $\norm{\mathcal{R}^C_K}=R_{K+1}N_{K+1}+\ldots+R_{M_0}N_{M_0}$. Also recall the definitions of $\q_K$ and $\mathbf{Q}_K$ from~\eqref{eq:04} and \eqref{eq:06}. For each $1 \leq K \leq M_0-1$ let
\begin{equation}\label{eq:09}
	A_K:= \left( \frac{\log n}{\log m}-1\right) \frac{\norm{\mathcal{R}^C_K}}{|\mathcal{R}_K|}.
\end{equation}

\begin{prop}\label{prop:opt}
Assume the parameters $n,m,N_1,\ldots,N_{M_0},R_1,\ldots,R_{M_0}$ define a Bedford--McMullen carpet $\Lambda$ with non-uniform vertical fibres. Then
$$\min_{\mathbf{p}} \dim_{\mathrm{M}}\nu_{\p} = \min\big\{ \alpha_1,\alpha_2,\ldots, \alpha_{M_0-1}\big\},$$ where for $1 \leq K \leq M_0-1$,
\begin{equation*}
\alpha_K:=\min\big\{\alpha(\mathbf{q}_K),\alpha(\mathbf{Q}_K),\alpha(\mathbf{q}_{K+1})\big\} =
\begin{cases}
\alpha(\mathbf{q}_K), & \text{if } A_K<N_K, \\
\alpha(\mathbf{Q}_K), & \text{if } N_K\leq A_K\leq N_{K+1}, \\
\alpha(\mathbf{q}_{K+1}), & \text{if } A_K>N_{K+1}.
\end{cases}
\end{equation*}
\end{prop}

We note that $\alpha(\mathbf{q}_K)=\alpha(\mathbf{Q}_K)$ if $A_K=N_K$, similarly, $\alpha(\mathbf{q}_{K+1})=\alpha(\mathbf{Q}_K)$ if $A_K=N_{K+1}$. Before turning to the proofs, we give examples when $M_0=2$ or $3$ and when the optimiser solving $\min_{\mathbf{p}} \dim_{\mathrm{M}}\nu_{\p}$ is not unique.


\subsection{Proof of Theorem~\ref{thm:3} and Proposition~\ref{prop:opt}}

Recall from Claim~\ref{claim:1} that it is enough to consider a vector $\mathbf{q}=(q_1, \ldots, q_{M_0})$ which satisfies~\eqref{eq:05} for some $1\leq K \leq M_0-1$. Recall that $|\mathcal{R}_K|= R_1+\ldots+R_K$ and $\norm{\mathcal{R}^C_K}=R_{K+1}N_{K+1}+\ldots+R_{M_0}N_{M_0}$. Therefore, from~\eqref{eq:05} we can express $q_{K+1}/N_{K+1}$ in terms of $q_K$ using that
\begin{equation}\label{k+1}
1=\sum_{\ell=1}^{M_0} R_{\ell}q_{\ell} = |\mathcal{R}_K|q_K + \norm{\mathcal{R}^C_K}\frac{q_{K+1}}{N_{K+1}} \;\;\Longleftrightarrow\;\; \frac{q_{K+1}}{N_{K+1}} = \frac{1-|\mathcal{R}_K| q_K}{\norm{\mathcal{R}^C_K}}.
\end{equation}
Combining~\eqref{eq:05} and~\eqref{k+1} yields the inequality
\begin{equation*}
\frac{q_{K}}{N_{K+1}} \stackrel{\eqref{eq:05}}{\leq} \frac{q_{K+1}}{N_{K+1}} \stackrel{\eqref{k+1}}{=} \frac{1-|\mathcal{R}_K| q_K}{\norm{\mathcal{R}^C_K}}  \stackrel{\eqref{eq:05}}{\leq} \frac{q_K}{N_K}.
\end{equation*}
After rearranging and using that $|\mathcal{R}_{K+1}|=|\mathcal{R}_{K}|+R_{K+1}$ and $\norm{\mathcal{R}^C_{K+1}} = \norm{\mathcal{R}^C_{K}}-R_{K+1}N_{K+1}$, we get the condition
\begin{equation} \label{range}
\frac{N_K}{N_K |\mathcal{R}_K|+\norm{\mathcal{R}^C_K}} \leq q_K \leq \frac{N_{K+1}}{N_{K+1} |\mathcal{R}_{K+1}|+\norm{\mathcal{R}^C_{K+1}}}.
\end{equation}
Also, substituting \eqref{k+1} back into $\alpha(\mathbf{q})$, recall~\eqref{eq:08}, we obtain a one variable function in $q_K$:
\begin{equation*}
f_K(q_K) := \frac{\log \frac{1-|\mathcal{R}_K| q_K}{\norm{\mathcal{R}^C_K}}}{-\log n} + \left(\! 1-\frac{\log m}{\log n}\right)\! \frac{\log q_{K}}{-\log m}.
\end{equation*}
Therefore, to obtain a solution to $\min_{\mathbf{q}}\alpha(\mathbf{q})$ it is enough to minimise $f_K(q_K)$ subject to condition~\eqref{range}.

To minimise $f_K(q_K)$, first observe that the equation $\frac{\mathrm{d}}{\mathrm{d} q_K}f_K(q_K) = 0$ yields the unique solution $q_K^*= \big( 1-\frac{\log m}{\log n}\big)/|\mathcal{R}_K|$. This is indeed a minimum, since the second derivative
\begin{equation*}
\frac{\mathrm{d}^2}{\mathrm{d} (q_K)^2}f_K(q_K)  = \frac{|\mathcal{R}_K|^2}{(1-|\mathcal{R}_K| q_K)^2\log n} + \left(\! 1-\frac{\log m}{\log n}\right)\! \frac{1}{(q_{K})^2\log m}  >0
\end{equation*}
for any $q_K$, in particular, also for $q_K=q_K^*$. Hence,
\begin{equation}\label{eq:10}
f_K(q_K) \text{ strictly decreases for } q_K<q^*_K \text{ and } f_K(q_K) \text{ strictly increses for } q_K>q^*_K.
\end{equation}

Recall $A_K$ from~\eqref{eq:09}. The condition $N_K \leq A_K \leq N_{K+1}$ from Proposition~\ref{prop:opt} is equivalent to
$$\frac{N_K}{N_K |\mathcal{R}_K|+\norm{\mathcal{R}^C_K}} \leq \left( 1-\frac{\log m}{\log n}\right)\frac{1}{|\mathcal{R}_K|} \leq \frac{N_{K+1}}{N_{K+1} |\mathcal{R}_{K+1}|+\norm{\mathcal{R}^C_{K+1}}},$$
which can be seen  by rearranging and using again that $|\mathcal{R}_{K+1}|=|\mathcal{R}_{K}|+R_{K+1}$ and $\norm{\mathcal{R}^C_{K+1}} = \norm{\mathcal{R}^C_{K}}-R_{K+1}N_{K+1}$. In particular, this implies that the global minimum $q_K^*$ for $f_K$ satisfies the bounds in \eqref{range}. Substituting $q_K=\left( 1-\frac{\log m}{\log n}\right)/|\mathcal{R}_K|$ into~\eqref{k+1} we recover the measure $\mathbf{Q}_K$ defined in~\eqref{eq:04}. Therefore, if $N_K \leq A_K \leq N_{K+1}$ then
$$\min_{\substack{\mathbf{q}:\\  \textnormal{\eqref{eq:05}holds for $K$}}}\alpha(\mathbf{q})=\alpha(\mathbf{Q}_K).$$

Next, note that if $A_K< N_K$ then similarly to the above, we can deduce that the global minima of $f_K$ satisfies
$$q^*_K=\left( 1-\frac{\log m}{\log n}\right)\frac{1}{|\mathcal{R}_K|}<\frac{N_K}{N_K |\mathcal{R}_K|+\norm{\mathcal{R}^C_K}}.$$
In particular, the global minima for $f_K$ is not in the range determined by~\eqref{range}. Instead, \eqref{eq:10} implies that the minimum of $f_K$ subject to~\eqref{range} is obtained at $q_K=\frac{N_K}{N_K |\mathcal{R}_K|+\norm{\mathcal{R}^C_K}}$. Substituting $q_K=\frac{N_K}{N_K |\mathcal{R}_K|+\norm{\mathcal{R}^C_K}}$ into \eqref{k+1}, we recover the measure $\mathbf{q}_K$ defined in \eqref{eq:04}. Hence, if $A_K< N_K$ then
$$\min_{\substack{\mathbf{q}:\\  \textnormal{\eqref{eq:05}holds for $K$}}}\alpha(\mathbf{q})=\alpha(\mathbf{q}_K).$$

Finally, if $A>N_{K+1}$ then the global minima of $f_K$ satisfies
$$q_K^*=\left( 1-\frac{\log m}{\log n}\right)\frac{1}{|\mathcal{R}_K|} >\frac{N_{K+1}}{N_{K+1} |\mathcal{R}_{K+1}|+\norm{\mathcal{R}^C_{K+1}}}.$$
This time \eqref{eq:10} implies that $q_K=\frac{N_{K+1}}{N_{K+1} |\mathcal{R}_{K+1}|+\norm{\mathcal{R}^C_{K+1}}}$ is where the minimum of $f_K$ subject to \eqref{range} is attained. Substituting this $q_K$ into \eqref{k+1}, we recover the measure $\mathbf{q}_{K+1}$. Hence, if $A>N_{K+1}$ then
$$ \min_{\substack{\mathbf{q}:\\  \textnormal{\eqref{eq:05}holds for $K$}}}\alpha(\mathbf{q})=\alpha(\mathbf{q}_{K+1}).$$

To conclude, we deduce from Claim~\ref{claim:1}  that $\min_{\mathbf{p}} \dim_{\mathrm{M}}\nu_{\p} = \min\big\{ \alpha_1,\alpha_2,\ldots, \alpha_{M_0-1}\big\}$. In particular by Claim \ref{claim:1} item (2) we can deduce the form that the optimising vector $\mathbf{p}^*$ for \eqref{eq:01} takes.
This completes the proof of Theorem~\ref{thm:3} and Proposition~\ref{prop:opt}.


\subsection{Special case with two different columns}\label{sec:twocols}

The input parameters of a Bedford--McMullen carpet with two different columns are: $n>m, N_1<N_2\leq n$ and $R_1+R_2\leq m$. Assume the indices of the maps $f_i$ defining the carpet are ordered such that the first $R_1N_1$ belong to columns with $N_1$ rectangles. By Proposition~\ref{prop:opt},
$$\min_{\mathbf{p}}\dim_{\mathrm{M}}\nu_{\p}= \alpha_1=\min\{\alpha(\mathbf{q}_1), \alpha(\mathbf{q}_2), \alpha(\mathbf{Q}_1)\},$$
therefore the optimising vector is either the uniform measure $\mathbf{p}_1=(1/N,\ldots,1/N)$, the coordinate uniform measure $\mathbf{p}_2=(1/(MN_{\psi(1)}),\ldots,1/(MN_{\psi(N)}))$ or $\mathbf{P}_1$ defined by
\begin{equation*}
P_{1,i}  = \left( 1-\frac{\log m}{\log n}\right)\frac{1}{R_1N_1}\text{ for all } i\leq R_1N_1 \;\;\text{ and }\;\; P_{1,i}  = \frac{\log m}{\log n}\frac{1}{R_2N_2}\text{ for all } i > R_1N_1.
\end{equation*}
In this case $A_1$ is given by
\begin{equation*}
A_1 = \left( \frac{\log n}{\log m}-1\right) \frac{R_2N_2}{R_1}.
\end{equation*}
Table~\ref{table:1} shows three examples where $\mathbf{p}_1,\, \mathbf{P}_1 $ and $\mathbf{p}_2$ are the optimising vectors, respectively. Note that in the last example, $\mathbf{q}_2$ is the optimizer regardless of the choice of $1\leq N_1<N_2\leq n$.

\begin{table}[H]
	\begin{center}
		\begin{tabular}{c|cccccccccc}
			\toprule
			& $R_1$ & $R_2$ & $N_1$ & $N_2$ & $m$ & $n$ & $\alpha(\mathbf{q}_1)$ & $\alpha(\mathbf{Q}_1)$ & $\alpha(\mathbf{q}_2)$ & $\dim_{\mathrm{M}}\Lambda$ \\ \midrule
	\rowcolor[HTML]{E8E8E8}
			$\phantom{N'_2<\,} A_1 < N_1$ & 1 & 1 & 2 & 3 & 2 & 3 & 1.95286 & -- & 2 & 1.83404 \\
			$N_1 \leq A_1 \leq N_2$ &1 & 1 & 1 & 2 & 2 & 3 & 1.58496 & 1.58089 & 1.63093 & 1.36907 \\
			\rowcolor[HTML]{E8E8E8}
			$N_2<A_1\phantom{>N'_2}$ & 1 & 1 & 2 & 3 & 2 & 5 & 1.75260 & -- & 1.68261 & 1.56932 \\
			\bottomrule
		\end{tabular}
	\end{center}
	\caption{Three examples in which each minimises $\dim_{\mathrm{M}}\nu_{\mathbf{p}}$ for a different vector.}
	\label{table:1}
\end{table}

\subsection{Simulations}\label{sec:simulations}

We demonstrate how to keep track of the orbit and how the choice of the measure $\mu$ driving the chaos game influences the ``quality" of the image on one of the examples presented in Table~\ref{table:1}.

For any $r>0$ consider a maximal $r$-packing of the attractor $\Lambda$, i.e. a collection of sets of diameter $r$ with disjoint interiors that cover $\Lambda$. The orbit $\mathcal{O}_n(\ii,x_0)$ becomes $r$-dense in $\Lambda$ once it has visited all elements of the $r$-packing. For Bedford--McMullen carpets we can keep track of the orbit using the collection of symbolic approximate squares of level $K$ introduced in~\eqref{eq:approxsquare}, where $K$ is chosen so that $m^{-K}\leq r < m^{-K+1}$. Recall, each approximate square $B_K(\ii)$ is identified with the sequence $(i_1,\ldots i_{L(K)};\phi(i_{L(K)+1}),\ldots,\phi(i_K))$, where $\phi(i)\in\{1,\ldots,M\}$ denotes the index of the column to which the rectangle $f_i([0,1]^2)$ belongs to. One step of the chaos game corresponds to the transition
\begin{equation*}
(i_1,\ldots i_{L(K)};\phi(i_{L(K)+1}),\ldots,\phi(i_K)) \longmapsto (j,i_1,\ldots i_{L(K)-1};\phi(i_{L(K)}),\ldots,\phi(i_{K-1}))
\end{equation*}
if $j$ was the next chosen index.

Table~\ref{table:2} shows the cover times in the middle example of Table~\ref{table:1} for various vectors at two different scales $K=6$ and $K=9$. The cover times are averaged out over 400 independent runs of the chaos game when $K=6$, and averaged out over 100 independent runs for $K=9$. The vectors $\p_1,\, \mathbf{P}_1 $ and $\mathbf{p}_2$ are the same as in Section~\ref{sec:twocols}, while  $\widehat{\p}$ corresponds to the McMullen measure which maximises the Hausdorff dimension of $\nu_{\mathbf{p}}$ defined as
\begin{equation*}
\widehat{p}_k= N_{\ih}^{\frac{\log m}{\log n}-1} \cdot \Big(\sum_{\jh=1}^M N_{\jh}^{\frac{\log m}{\log n}}\Big)^{-1} \;\text{ if } \phi(k)=\ih.
\end{equation*}
The pair $(\widetilde{\p},\widetilde{\q})$ corresponds to the uniform vectors on $\{1,\ldots,N\}$ and $\{1,\ldots,M\}$, respectively. In this case, we modified the chaos game to ``two dimensions": in a transition step a new map \emph{and} a new column are chosen \emph{independently} of each other according to $\widetilde{\p}$ and $\widetilde{\q}$. That is, if $k\in\{1,\ldots,N\}$ and $\ell\in\{1,\ldots,M\}$ are chosen uniformly and independently, then a transition step is
\begin{equation*}
(i_1,\ldots i_{L(K)};j_{L(K)+1},\ldots,j_K) \longmapsto (k,i_1,\ldots i_{L(K)-1};\ell,j_{L(K)+1},\ldots,j_{K-1}).
\end{equation*}
The choice $(\widetilde{\p},\widetilde{\q})$ is optimal in the sense that it minimises
\begin{equation*}
\alpha(\p,\q) := \max_{\substack{1\leq\, i\,\leq N \\ 1\leq\, \jh\, \leq M}} \left\{ \frac{\log p_i}{-\log n} + \left( 1-\frac{\log m}{\log n}\right) \frac{\log q_{\jh}}{-\log m} \right\}
\end{equation*}
with value $\alpha(\widetilde{\p},\widetilde{\q})=\dim_{\mathrm{M}}\Lambda$. Hence, is the most efficient possible. However, the drawback of this modified chaos game is that it does not have any clear geometric interpretation.

Indeed, Table~\ref{table:2} shows that the runtime is substantially faster with $(\widetilde{\p},\widetilde{\q})$ than with the other vectors. Table~\ref{table:1} suggests that the runtime with $\mathbf{P}_1$ should be second fastest since it corresponds to the smallest exponent $\alpha$, however, this is not supported by the empirical data in Table~\ref{table:2}. One explanation for this could be the fact that for small values of $K$ the measure of $B_K(\ii^*)$, which denotes the level $K$ approximate square of minimum measure, does not yet reflect the asymptotic behaviour inferred from the exponent $\alpha$.
For example, by direct computation we obtain that at level $K=6$
\begin{equation*}
0.00137=\mu_{\p_1}(B_6(\ii^*)) < \mu_{\mathbf{P}_1}(B_6(\ii^*)) < \mu_{\widehat{\p}}(B_6(\ii^*)) < \mu_{\p_2}(B_6(\ii^*))=0.00195,
\end{equation*}
whereas for level $K=100$
\begin{equation*}
8.55\times10^{-50}=\mu_{\p_2}(B_K(\ii^*)) < \mu_{\p_1}(B_K(\ii^*)) < \mu_{\widehat{\p}}(B_K(\ii^*)) < \mu_{\mathbf{P}_1}(B_K(\ii^*)) = 2.61\times10^{-48}.
\end{equation*}
Another reason could be that the impact of the subexponential error terms on the runtime is amplified for small values of $K$. For larger $K$, we expect the runtime would follow the order of the exponents but, since the number of approximate squares grows exponentially in $K$, it is computationally impossible to do simulations for much larger $K$.

\begin{table}[h]
	\begin{center}
		\begin{tabular}{l|ccccccrrrrr}
			\toprule
			& $R_1$ & $R_2$ & $N_1$ & $N_2$ & $m$ & $n$ & $\mathbf{p}_1$ & $\mathbf{P}_1$ & $\mathbf{p}_2$ & $\widehat{\p}$ & $(\widetilde{\p},\widetilde{\q})$ \\ \midrule
			\rowcolor[HTML]{E8E8E8}
			$K=6$ & 1 & 1 & 1 & 2 & 2 & 3 & 2787 & 2202 & 2442 & 2060 & 1288 \\
			$K=9$ &1 & 1 & 1 & 2 & 2 & 3 & 118057 & 86445 & 112666 & 78910 & 33855 \\
			\bottomrule
		\end{tabular}
	\end{center}
	\caption{The runtime of the chaos game on the same example with different vectors until it visits all approximate squares at level $K$.}
	\label{table:2}
\end{table}

In Figure~\ref{fig:BMCarpetSimu} we plotted the orbit starting from $(1,1)\in\Lambda$ using three different vectors; from left to right $\mathbf{P}_1,\p_1$ and $(0.6,0.25,0.15)$. The orbits were terminated once the game with $\mathbf{P}_1$ visited all approximate squares at level $K=7$. It is difficult to see any difference between the first two figures with the naked eye. However, in the third case where the vector was chosen deliberately to be very different from the optimal, it is apparent that it hasn't visited many approximate squares in the same number of steps.

\begin{figure}[h]
	\centering
	\includegraphics[width=0.97\textwidth]{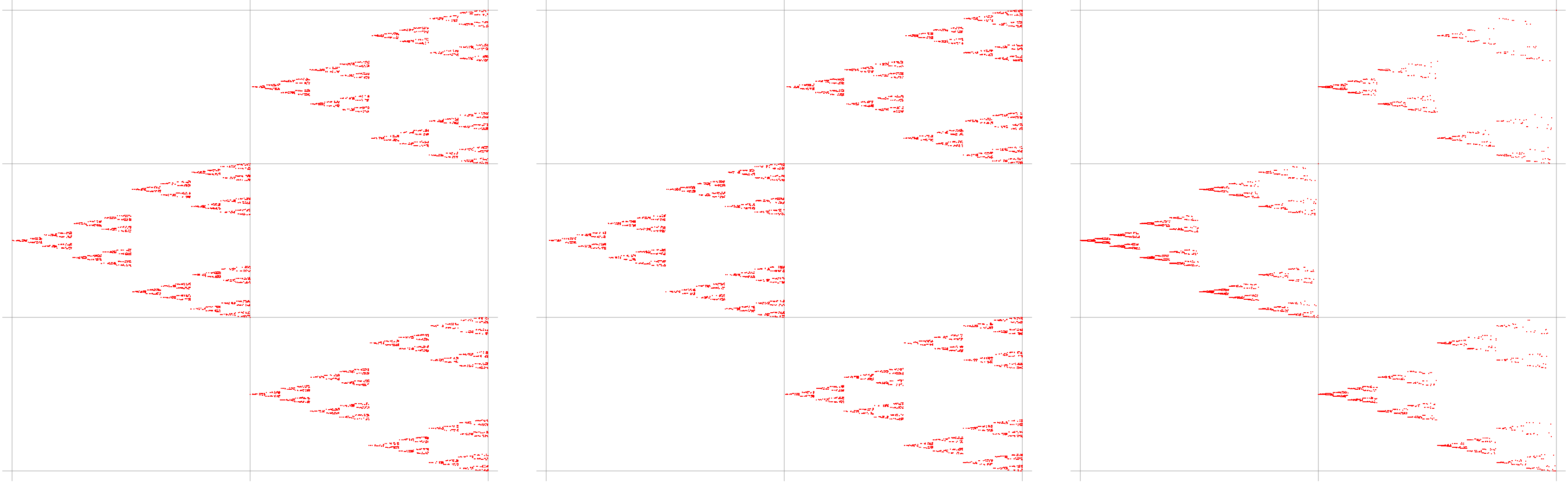}
	\caption{Plots of orbits using three different vectors (from left to right $\mathbf{P}_1,\p_1$ and $(0.6,0.25,0.15)$) terminated when the game with $\mathbf{P}_1$ visited all level-7 approximate squares.}
	\label{fig:BMCarpetSimu}
\end{figure}

\subsection{Non-unique optimiser} \label{sec:nonuniqueopt}

Let us again consider a Bedford--McMullen carpet with two different columns, where $m=2$, $n=4$, $R_1=R_2=1$ and $N_1=2$, $N_2=4$. Assume the indices of the maps $f_i$ defining the carpet are ordered such that the first $2$ belong to the column that contains 2 rectangles. We know that $\min_{\mathbf{p}}\dim_{\mathrm{M}}\nu_{\p}=\alpha_1=\alpha(\mathbf{Q}_1)$ since $A_1=4=N_2$. Moreover $\mathbf{Q}_1=\q_2=\big(\frac{1}{2}, \frac{1}{2}\big)$, meaning that the co-ordinate uniform vector
$$\mathbf{P}_1=\Big(\frac{1}{4}, \frac{1}{4}, \frac{1}{8},\frac{1}{8},\frac{1}{8},\frac{1}{8}\Big)$$
is an optimiser for \eqref{eq:01}.

Now, for $0 \leq\varepsilon\leq \frac{1}{8}$ define the perturbed vector
$$\mathbf{P}^\varepsilon_1=\Big(\frac{1}{4}-\varepsilon, \frac{1}{4}+\varepsilon, \frac{1}{8},\frac{1}{8},\frac{1}{8},\frac{1}{8}\Big).$$
From \eqref{eq:02} one can check that $\dim_{\mathrm{M}}\nu_{\mathbf{P}_1^\varepsilon}$ is actually independent of $\varepsilon$, thus each $\mathbf{P}_1^\varepsilon$ is also an optimising vector for $0 \leq\varepsilon\leq \frac{1}{8}$. In particular, the optimising vector for \eqref{eq:01} is not unique.

\subsection{Special case with three different columns}

The input parameters of a Bedford--McMullen carpet with three different columns are: $n>m, N_1<N_2<N_3\leq n$ and $R_1+R_2+R_3\leq m$. The parameter space, shown in Figure~\ref{fig:parameterspace}, consists of vectors $\mathbf{q}=(q_1,q_2,q_3)$ such that $R_1q_1+R_2q_2+R_3q_3=1$. It may be that $N_2/N>1/M$, however, the line connecting $\mathbf{q}_1$ to $(1/R_1,0)$ always intersects the line $q_1=q_2$ (defining $\mathbf{q}_2$) and never the line connecting $\mathbf{q}_3$ with $(0,1/R_2)$.

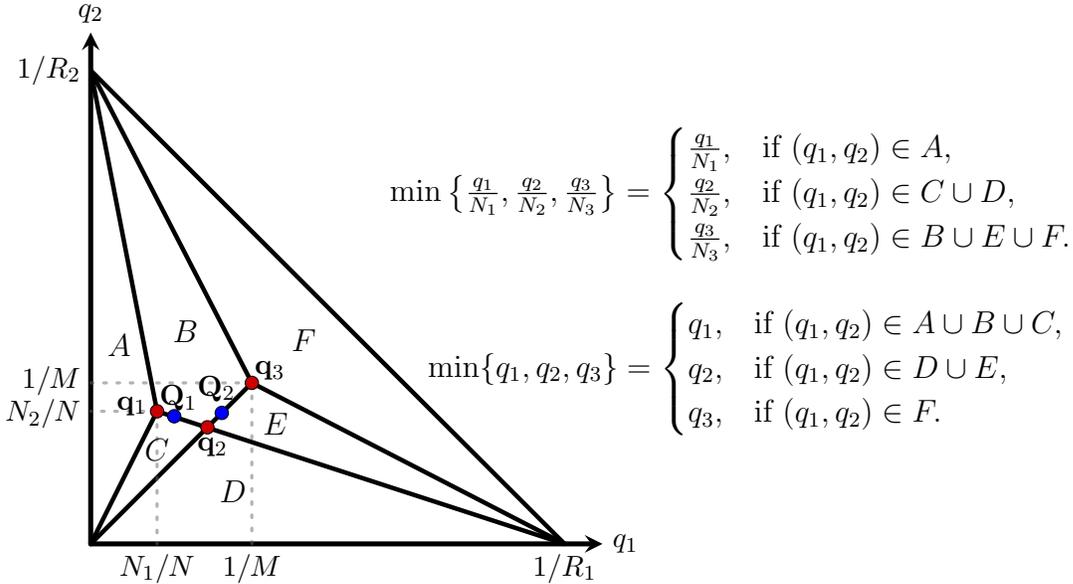
\begin{figure}[h]
\begin{tikzpicture}[scale=0.63,line cap=round,line join=round,>=stealth,x=1cm,y=1cm]
\clip(-2,-1) rectangle (25,12);
\draw [->,line width=1.9pt] (0,0) -- (10.8,0);
\draw [->,line width=1.9pt] (0,0) -- (0,10.8);
\draw [line width=1.7pt] (0,10)-- (10,0);
\draw [line width=1.7pt] (3.4,3.4)-- (0,10);
\draw [line width=1.7pt] (3.4,3.4)-- (10,0);
\draw [line width=1.7pt] (3.4,3.4)-- (0,0);
\draw [line width=1.2pt,dash pattern=on 1pt off 4pt,color=wwwwww, opacity=0.5] (3.4,3.4)-- (3.4,0);
\draw [line width=1.2pt,dash pattern=on 1pt off 4pt,color=wwwwww, opacity=0.5] (3.4,3.4)-- (0,3.4);
\draw [line width=1.7pt] (1.4,2.8)-- (10,0);
\draw [line width=1.7pt] (1.4,2.8)-- (0,0);
\draw [line width=1.7pt] (1.4,2.8)-- (0,10);
\draw [line width=1.2pt,dash pattern=on 1pt off 4pt,color=wwwwww, opacity=0.5] (1.4,2.8)-- (1.4,0);
\draw [line width=1.2pt,dash pattern=on 1pt off 4pt,color=wwwwww, opacity=0.5] (1.4,2.8)-- (0,2.8);
\draw (10,0) node[anchor=north] {$1/R_1$};
\draw (3.4,0) node[anchor=north] {$1/M$};
\draw (1.4,0) node[anchor=north] {$N_1/N$};
\draw (0,2.75) node[anchor=east] {$N_2/N$};
\draw (0,3.45) node[anchor=east] {$1/M$};
\draw (0,10) node[anchor=east] {$1/R_2$};

\draw (0.6,4.2) node {\large{$A$}};
\draw (2,4.5) node {\large{$B$}};
\draw (1.4,2) node {\large{$C$}};
\draw (3,1.1) node {\large{$D$}};
\draw (3.9,2.53) node {\large{$E$}};
\draw (4.5,4.3) node {\large{$F$}};

\draw (10.8,0) node[anchor=west] {$q_1$};
\draw (0,10.8) node[anchor=south] {$q_2$};

\begin{scriptsize}
\draw [fill=ccqqqq] (3.4,3.4) circle (4pt);
\draw [fill=ccqqqq] (1.4,2.8) circle (4pt);
\draw [fill=ccqqqq] (2.462,2.462) circle (4pt);
\draw [fill=qqqqff] (1.762,2.6952) circle (4pt);
\draw [fill=qqqqff] (2.766,2.766) circle (4pt);
\end{scriptsize}
\draw (1.4,2.5) node[anchor=south east] {$\mathbf{q}_1$};
\draw (2.57,2.45) node[anchor=north] {$\mathbf{q}_2$};
\draw (3.25,3.25) node[anchor=south west] {$\mathbf{q}_3$};
\draw (1.85,2.6) node[anchor=south] {$\mathbf{Q}_1$};
\draw (2.65,2.8) node[anchor=south] {$\mathbf{Q}_2$};
\draw (6.9,3.7) node[anchor=west] {\large{$\min\{q_1,q_2,q_3\} = \begin{cases}q_1, & \text{if } (q_1,q_2)\in A\cup B\cup C, \\ q_2, & \text{if } (q_1,q_2)\in D\cup E, \\ q_3, & \text{if } (q_1,q_2)\in F. \end{cases}$}};
\draw (6.1,7.4) node[anchor=west] {\large{$\min\big\{\frac{q_1}{N_1},\frac{q_2}{N_2},\frac{q_3}{N_3}\big\} = \begin{cases} \frac{q_1}{N_1}, & \text{if } (q_1,q_2)\in A, \\ \frac{q_2}{N_2}, & \text{if } (q_1,q_2)\in C\cup D, \\ \frac{q_3}{N_3}, & \text{if } (q_1,q_2)\in B\cup E\cup F. \end{cases}$}};

\end{tikzpicture}
\caption{The parameter space of vectors $\mathbf{q}=(q_1,q_2,(1-R_1q_1-R_2q_2)/R_3)$ shown with the six regions where $\alpha(\mathbf{q})$ takes different values.}
\label{fig:parameterspace}
\end{figure}

The function $\alpha(\mathbf{q})$ takes different values according to which region the vector $\mathbf{q}$ falls into. The candidates that minimise $\alpha(\mathbf{q})$ are the vectors $\mathbf{q}_1,\mathbf{q}_2,\mathbf{q}_3,\mathbf{Q}_1$ and $\mathbf{Q}_2$. $\mathbf{Q}_1$ is a valid candidate if and only if it lies on the line between $\mathbf{q}_1$ and $\mathbf{q}_2$. Likewise for $\mathbf{Q}_2$ between $\mathbf{q}_2$ and $\mathbf{q}_3$. We leave it to the reader to find five different examples, where each one minimizes $\alpha(\mathbf{q})$ for a different candidate. We note that the McMullen measure which maximizes the Hausdorff dimension of $\nu_{\mathbf{p}}$ lies somewhere in the interior of region $B$.

\section{Open questions}\label{sec:6}

Here we suggest some possible directions for future investigation appearing naturally from the results of the presented paper.

\begin{ques} Can Theorem \ref{thm:2} be extended beyond the class of measures that satisfy the one-sided exponential decay of correlations property? For example, to measures which satisfy a weaker mixing property, or even to all ergodic measures?
\end{ques}

\begin{ques}
	In Proposition~\ref{prop:lower}, is it possible to omit the assumption that $\dim_{\mathrm{L}} \Lambda>0$ if the IFS is overlapping? In point of view of Proposition~\ref{prop:lower2} is it possible to relax the conditions even further?
\end{ques}

\begin{ques} Is it possible to relax the hyperbolicity condition in Theorem \ref{thm:2}? That is, is it possible to relax the statement to allow for parabolic iterated function systems $\mathcal{F}$, where  $\mathcal{F}$ is no longer uniformly contracting, but it contains a map with a neutral fixed point?
\end{ques}

\begin{ques} At least in special cases, as was shown in~\cite{JurgaMorris_ChaosGame_2020arxiv} for self-similar sets, is it possible to obtain tighter bounds on the expected value of the cover time in Theorem \ref{thm:expectedvalue}?
\end{ques}

\begin{ques}
Let us consider a simple iterated function system of similarities with the strong separation condition, and let us consider the natural self-similar measure $\mu$ defined by the similarity dimension, and fix an $r>0$. Is it possible to give tight bounds on $n$ for which $\mu(\{\ii:d_{\mathrm H} \big(\mathcal{O}_n(\ii,x_0),\Lambda\big)<r\})\geq0.95$? Is it possible in the case of the weak separation condition? Or even without any separation condition?
\end{ques}

\begin{ques}
In case of Bernoulli convolutions, recall~\eqref{eq:BernoulliConv}, all self-similar measures have Minkowski dimension strictly larger than 1. Can it be achieved if the probability vector $\mathbf{p}$ is allowed to dependent on the current position of the orbit? More precisely, let $p:\Lambda\to[0,1]$. Assuming some regularity on $p$, see~\cite{FanLau_JMAA}, there exists a unique measure $\nu$ satisfying
\begin{equation*}
\nu(B) = \int_{(B+1)/\lambda} p(x) \mathrm{d}\nu(x) + \int_{(B-1)/\lambda} (1-p(x)) \mathrm{d}\nu(x)
\end{equation*}
for any Borel set $B$. Can $\dim_{\mathrm{M}}\nu=1$ be achieved with an appropriate choice of $p$?
\end{ques}

\begin{ques} Recall that the Minkowski dimensions of Bernoulli measures supported on a Bedford-McMullen carpet $\Lambda$ obtain a maximum at $\dim_{\mathrm{M}} \Lambda$ if and only if $\Lambda$ has uniform vertical fibres (Claim \ref{claim:2}). Therefore it is interesting to ask whether, in the non-uniform vertical fibres case, there exists an invariant measure supported on $\Lambda$ whose Minkowski dimension equals  $\dim_{\mathrm{M}} \Lambda$ (thus which optimises the cover time for the chaos game)? Note that although it is known that $\dim_{\mathrm{M}} \Lambda$ is always achieved by the Minkowski dimension of some measure supported on $\Lambda$, this measure is not necessarily invariant \cite{FFF_MinkowskiDimMeas_2020arxiv}.
\end{ques}

\bibliographystyle{abbrv}
\bibliography{biblio_chaosgame}

\end{document}